\documentclass[12pt]{extarticle}
\usepackage{amsmath, amsthm, amssymb, hyperref, color}
\usepackage[shortlabels]{enumitem}
\usepackage{graphicx}
\usepackage[all]{xypic}
\usepackage{makecell}
\usepackage[final]{pdfpages}
\setboolean{@twoside}{false}
\usepackage{pdfpages}
\usepackage{caption}
\usepackage{subcaption}
\usepackage{scalefnt}
\usepackage{verbatim}
\oddsidemargin \evensidemargin
\newtheorem{theorem}{Theorem}
\numberwithin{theorem}{section}
\newtheorem{proposition}[theorem]{Proposition}
\newtheorem{lemma}[theorem]{Lemma}
\newtheorem{corollary}[theorem]{Corollary}

\newtheorem{problem}[theorem]{Problem}
\newtheorem{remark}[theorem]{Remark}
\newtheorem{example}[theorem]{Example}
\newtheorem{conjecture}[theorem]{Conjecture}
\newtheorem{algorithm}[theorem]{Algorithm}

\newcommand{\RR}{\mathbb{R}}

\newcommand{\PP}{\mathbb{P}}
\newcommand{\CC}{\mathbb{C}}
\newcommand{\ZZ}{\mathbb{Z}}
 
\newcommand{ \rrk }{{\mathrm{rk}_{\mathbb{R}}}}
\newcommand{ \crk }{{\mathrm{rk}_{\mathbb{C}}}}
 \date{}
\newcommand{\SSS}{S}
\newcommand{\BQ}{\Omega}
 
\title{\textbf{Real Rank Geometry of Ternary Forms}}

\author{Mateusz Micha{\l}ek, Hyunsuk Moon, \\ 
Bernd Sturmfels, and Emanuele Ventura}

\begin{document}

\maketitle

\begin{abstract} \noindent
We study real ternary forms whose real rank equals the generic complex rank,
and we characterize the semialgebraic set of 
sums of powers representations with that rank.
Complete results are obtained for quadrics and cubics. For quintics we determine the real rank boundary: it is a hypersurface of degree $168$.
For quartics, sextics and septics we identify some of 
the components of the real rank boundary. The real varieties of sums of powers are
stratified by discriminants that are derived from hyperdeterminants.

\end{abstract}

\section{Introduction}

Let $ \RR[x,y,z]_d$ denote the $\binom{d+2}{2}$-dimensional
vector space of ternary forms $f$ of degree $d$. These are homogeneous
polynomials of degree $d$ in three unknowns $x,y,z$, or equivalently,
symmetric tensors of format $3 {\times} 3 {\times} \cdots {\times} 3$ with $d$ factors.
We are interested in decompositions
\begin{equation}
\label{eq:waring}
 f(x,y,z) \,\, = \,\, \sum_{i=1}^r \,\lambda_i \cdot (a_i x + b_i y + c_i z)^d, 
 \end{equation}
where $\,a_i,b_i,c_i, \lambda_i \in \RR\,$ for $\,i=1,2,\ldots,r$.
The smallest $r$ for which such a representation exists is
the {\em real rank} of $f$, denoted $\rrk(f)$. 
For $d$ even, the representation (\ref{eq:waring}) has {\em signature}
$(s,r-s)$, for $s \geq r/2 $, if $s$ of the $\lambda_i$'s are positive
while the others are negative, or vice versa.
The {\em complex rank} $\crk(f)$ is the smallest $r$
such that $f$ has the form
(\ref{eq:waring}) where  $a_i,b_i,c_i \in \CC$.
The inequality  $\crk(f) \leq \rrk(f)$ always holds, and is often strict.
For binary forms,  this phenomenon is well-understood by now, thanks to
 \cite{Ble, CO}.  For ternary forms,
 explicit regions where the inequality is strict
 were identified  by
  Blekherman, Bernardi and Ottaviani  in \cite{BBO}.
 
The present paper extends these studies. We focus on
ternary forms $f$ that are general in $ \RR[x,y,z]_d$.
The complex rank  of such a form is referred to as the
{\em generic rank}. It depends only on $d$,  and we denote it by  $ R(d) $.
The Alexander-Hirschowitz Theorem \cite{AH} implies that
\begin{equation}
\label{eq:genrank}
R(2) = 3,\,\, R(4) = 6, \,\, {\rm and} \quad
R(d) \, = \,\bigg\lceil \frac{ (d+2)(d+1)}{6} \bigg\rceil \,\,\, {\rm otherwise}.
\end{equation}
We are particularly interested in
general forms whose minimal decomposition is real. Set
$$  \mathcal{R}_d \,= \,
\bigl\{\, f \in \RR[x,y,z]_d \,: \,
\rrk(f) = R(d) \bigr\}. $$
This is a full-dimensional semialgebraic subset of $\RR[x,y,z]_d$.
Its {\em topological boundary} $\partial \mathcal{R}_d$ is
the set-theoretic difference of the closure of $\mathcal{R}_d$
minus the interior of the closure of $\mathcal{R}_d$.
Thus, if $ f \in \partial \mathcal{R}_d$
then every open neighborhood of $f$
contains a general form
of real rank equal to  $R(d)$
and also a general form of real rank bigger than $R(d)$.
The semialgebraic set  $\partial \mathcal{R}_d $ is 
either empty or pure of codimension $1$.
The {\em real rank boundary}, denoted
$\partial_{\rm alg}(\mathcal{R}_d)$, is defined as the Zariski closure
of the topological boundary $\partial \mathcal{R}_d$
in the complex projective space $\PP(\CC[x,y,z]_d) = \PP^{\binom{d+2}{2}-1}$.
We conjecture that the variety $\partial_{\rm alg}(\mathcal{R}_d)$ is non-empty, and  hence has codimension $1$, for all $d \geq 4$. This is equivalent to 
$R(d)+1$ being a typical rank, in the sense of \cite{BBO, Ble, CO}.
This is proved for $d=4,5$ in \cite{BBO}
and for $d=6,7,8$ in this paper.

Our aim is to study  these hypersurfaces. The big guiding problem is as follows:

\begin{problem} \label{prob:big}
Determine the polynomial that defines the
real rank boundary $\partial_{\rm alg}(\mathcal{R}_d)$.
\end{problem}

The analogous question for binary forms was answered in
\cite[Theorem 4.1]{LS}. A related and equally difficult issue is to identify 
all the various open strata in the real rank stratification.

\begin{problem}
Determine the possible real ranks of general ternary forms in $\,\RR[x,y,z]_d$.
\end{problem}

This problem is open for $d \geq 4$;
the state of the art is the
work of Bernardi, Blekherman and Ottaviani in \cite{BBO}.
For binary forms, this question has a complete answer, due to
Blekherman \cite{Ble}, building on earlier work
of Comon and Ottaviani~\cite{CO}.
See also \cite[\S 4]{LS}.

For any ternary form $f$ and the generic rank $r = R(d)$, it is natural
to ask for the space of all decompositions (\ref{eq:waring}).
In the algebraic geometry literature \cite{Muk, RS},
this space is denoted ${\rm VSP}(f)$ and called
 the {\em variety of sums of powers}. 
 By definition, ${\rm VSP}(f)$ is the closure of the subscheme
of the Hilbert scheme ${\rm Hilb}_r(\PP^2)$
parametrizing the unordered configurations
\begin{equation}
\label{eq:rpoints}
 \bigl\{ (a_1:b_1:c_1),\, (a_2:b_2:c_2),\,
\ldots, (a_r:b_r:c_r) \bigr\} \,\,\subset \,\PP^2 
\end{equation}
that can occur in (\ref{eq:waring}).
If $f$ is general then the dimension of its variety of sums of powers
    depends only on $d$. By counting parameters, the Alexander-Hirschowitz Theorem \cite{AH} implies
\begin{equation}
\label{eq:dimVSP}
{\rm dim}({\rm VSP}(f)) \,\,=\,\,
\begin{cases}
\,\,3\,\,\text{ if }d=2\text{ or }4,\\
\,\,2\,\,\text{ if }d=0 \text{\,(mod }3),\\
\,\,0\,\,\text{ otherwise.}
\end{cases}
\end{equation}
In Table \ref{tab:VSPs of ternary forms}
 we summarize
 what is known 
about the varieties of sums of powers. In $2/3$ of all cases, the variety ${\rm VSP}(f)$ is finite.
It is one point only in the case of quintics, by  \cite{Mel}.

\begin{table}
\begin{tabular}{| l | l | l | l | l | r |}
  \hline
   Ternary forms  &  $R(d)$  & ${\rm VSP}(f)$ & Reference \\
     \hline 
  Quadrics & $3$ & del Pezzo threefold $V_5$ &  Mukai \cite{Muk92} \\
  \hline
   Cubics & $4$ & $\PP^2$ & Dolgachev, Kanev \cite{DK} \\ 
  \hline
  Quartics & $6$ & Fano threefold $V_{22}$ of genus $12$ & Mukai \cite{Muk92} \\
  \hline
Quintics & $7$ & $1$ point & Hilbert, Richmond, Palatini; see \cite{RS}\\
  \hline
 Sextics & $10$ & $K3$ surface $V_{38}$ of genus $20$ & Mukai \cite{Muk92b}; see also \cite{RS} \\
  \hline
 Septics & $12$ & $5$ points  & Dixon, Stuart \cite{DS} \\
 \hline
 Octics & $15$ & $16$ points & Ranestad, Schreyer \cite{RS}\\
 \hline
\end{tabular}
\caption{Varieties of sums of powers for ternary forms of degree
$d=2,3,4,5,6,7,8$.}
 \label{tab:VSPs of ternary forms}
\end{table}

We are interested in the semialgebraic subset ${\rm SSP}(f)_\RR$ 
of those configurations  (\ref{eq:rpoints})  in  ${\rm VSP}(f)$ 
 whose $r$ points all have real coordinates.
 This is the   {\em space of real sums of powers}.
Note that the space ${\rm SSP}(f)_\RR$ is non-empty
if and only if the ternary form $f $ lies in the semialgebraic set $\mathcal{R}_d$.
The inclusion of $\overline{{\rm SSP}(f)_\RR}$ in the real variety ${\rm VSP}(f)_\RR$ 
of real points of ${\rm VSP}(f)$
is generally strict. Our aim is to describe these objects as explicitly as possible.

A key player is the {\em apolar ideal} of the form $f$.
This is the $0$-dimensional Gorenstein ideal
\begin{equation}
\label{eq:keyplayer}
 f^\perp \,\, = \,\, \Bigl\{ \,p(x,y,z) \in \RR[x,y,z] \,:\,\,
p\bigl( \frac{\partial}{\partial x},
 \frac{\partial}{\partial y},  \frac{\partial}{\partial z}  \bigr)
 \,\,\,{\rm annihilates} \,\, f(x,y,z) \, \Bigr\}. 
 \end{equation}
 A configuration (\ref{eq:rpoints}) lies in ${\rm VSP}(f)$
 if and only if its homogeneous radical ideal is contained in $f^\perp$.
 Hence, points in ${\rm SSP}(f)_\RR$ are
$1$-dimensional radical ideals in $f^\perp$ whose zeros are real.
 
Another important tool is the  middle catalecticant of $f$,
which is defined as follows. 
For any partition $d = u+v$,  consider the bilinear form
$\, C_{u,v}(f) : \RR[x,y,z]_{u} \times \RR[x,y,z]_{v} \rightarrow \RR \,$
that maps
$(p,q)$ to the real number
obtained by applying
$(p \cdot q) ( \frac{\partial}{\partial x},
 \frac{\partial}{\partial y},  \frac{\partial}{\partial z}  ) $
 to the polynomial~$f$.
 We identify $C_{u,v}(f)$ with the matrix that represents the
 bilinear form with respect to the monomial basis.
 The {\em middle catalecticant} $C(f)$ of the ternary form $f$ is precisely that matrix,
 where we take $u=v = d/2$ when $d$ is even, and $u=(d{-}1)/2, v=(d{+}1)/2$
 when $d$ is odd.
The hypothesis $d \in \{2,4,6,8\}$ ensures that $C(f)$ is square 
of size equal to $R(d) = \binom{d/2+2}{2}$.

\begin{proposition}
\label{lem:2468}
Let $d \in \{2,4,6,8\}$ and $f \in \RR[x,y,z]_d$ be general.
The signature of any representation  (\ref{eq:waring})
coincides with the signature of 
the middle catalecticant $C(f)$.
If $C(f)$ is positive definite then  $\overline{\,{\rm SSP}(f)_\RR }= {\rm VSP}(f)_\RR $,
and this  set is always non-empty provided $d \leq 4$.
 \end{proposition}
 
 \begin{proof}
 If $f = \sum_{i=1}^r  \lambda_i \ell_i^d$ as in (\ref{eq:waring})
 then  $C(f)$ is the sum of the rank one matrices $\lambda_i C(\ell_i^d)$.
 If $C(f)$ has rank $r$ then its signature is $(\# \textnormal{ positive } \lambda_i, \#\textnormal{ negative }  \ \lambda_i)$. The identity
  ${\rm SSP}(f)_\RR = {\rm VSP}(f)_\RR $ will be proved for
    $d=2$ in  Theorem \ref{thm:quadrics}, and the same argument
  works for $d=4,6,8$ as well.
  The last assertion, for $d \leq 4$, is due to Reznick
  \cite[Theorem 4.6]{Rez}.
 \end{proof}

The structure of the paper is organized by increasing degrees:
Section $d$ is devoted to ternary forms of degree $d$.
In Section 2, we determine the threefolds ${\rm SSP}(f)_\RR$
for quadrics, and in Section 3 we determine the surfaces
${\rm SSP}(f)_\RR$ for cubics. Theorem~\ref{thm:diskmobius} summarizes 
the four cases displayed in  Table~\ref{tab:behavior}.
Section 4 is devoted to quartics $f$ and their real rank boundaries.
We present an algebraic characterization of ${\rm SSP}(f)_\RR$
as a subset of Mukai's  Fano threefold $V_{22}$, following \cite{KS,  Muk92,Muk, RS, Sch}.
In Section 5, we use the uniqueness of the rank $7$
decomposition of quintics to determine the irreducible
hypersurface $\partial_{\rm alg} (\mathcal{R}_5)$.
We also study the case of septics $(d=7)$, and we discuss
${\rm VSP}_X$ for arbitrary varieties $X \subset \PP^N$.
Finally, Section~6 covers all we know about sextics, starting in Theorem~\ref{thm:severi}
with a huge component of  the boundary $\partial_{\rm alg} (\mathcal{R}_6)$,
and concluding with a case study of the 
monomial $f = x^2 y^2 z^2$.

This paper contains numerous open problems and conjectures.
We are fairly confident about some of them
(like the one stated prior to Problem \ref{prob:big}). However,
 others (like Conjectures \ref{conj:3comp} and \ref{conj:star}) are based
primarily on optimism. We hope that all will be useful in inspiring further
 progress on the real algebraic geometry of tensor decompositions.

\section{Quadrics}

The real rank geometry of quadratic forms is surprisingly delicate 
and interesting.
Consider a general real quadric $f$ in $n$ variables.
We know from linear algebra that $\rrk(f) = \crk(f) = n$. More precisely,
if $(p,q)$ is the {\em signature} of $f$ then,
 after a linear change of coordinates,
\begin{equation}
\label{eq:pqquadric} \qquad
 f \, = \, x_1^2 + \cdots + x_p^2 - x_{p+1}^2 - \cdots - x_{p+q}^2
 \qquad \qquad (n = p+q).
 \end{equation}
 The stabilizer of $f$ in ${\rm GL}(n,\RR)$ 
 is denoted ${\rm SO}(p,q)$. It is called the {\em indefinite special
 orthogonal group} when $p,q \geq 1$. 
 We denote by ${\rm SO}^{+}(p,q)$ the connected component of ${\rm SO}(p,q)$  containing the identity.  Let $G$ denote the stabilizer in ${\rm SO}^{+}(p,q)$ of the set 
 $\bigl\{\{x_1^2,\ldots,x_p^2\}, \{x_{p+1}^2,\ldots,x_n^2 \}\bigr\}$.
 In particular, if $f$ is positive definite then
 we get the group of rotations,
    ${\rm SO}^{+}(n,0) = {\rm SO}(n)$, 
  and $G$ is the subgroup of rotational symmetries
  of the $n$-cube, which has order $2^{n-1} n!$.
  
 \begin{theorem} \label{thm:quadrics}
 Let $f$ be a rank $n$ quadric of signature $(p,q)$.
 The space ${\rm SSP}(f)_\RR$ can be identified with the quotient ${\rm SO}^{+}(p,q)/G$.
 If the quadric $f$ is definite then
  ${\rm SSP}(f)_\RR = {\rm VSP}(f)_\RR =  {\rm SO}(n)/G$.
  In all other cases, $\overline{{\rm SSP}(f)_\RR}$
  is strictly contained in the real variety
  ${\rm VSP}(f)_\RR$.
   \end{theorem}

\begin{proof}
The analogue of the first assertation over an algebraically closed field
appears in \cite[Proposition 1.4]{RSpolar}.
To prove ${\rm SSP}(f)_\RR = {\rm SO}^{+}(p,q)/G$ over $\RR$,
we argue as follows. Every rank $n$ decomposition of $f$ has the form
$\sum_{i=1}^p \ell_i^2-\sum_{j=p+1}^{p+q} \ell_j^2$, and is hence obtained from (\ref{eq:pqquadric})
 by an invertible linear transformation $x_j \rightarrow \ell_j $ that preserves $f$. These 
 elements of ${\rm GL}(n,\RR)$ are taken
  up to sign reversals and permutations of the sets
 $\{\ell_1,\ldots,\ell_p\}$ and $\{\ell_{p+1},\ldots,\ell_{n}\}$.
 
 Suppose that $f$ is not definite, i.e.~$p,q \geq 1$.
 Then we can write
 $ f = 2\ell_1^2-2\ell_2^2+\sum_{j=3}^n \pm \ell_j^2$. 
 Over $\CC$, with $i = \sqrt{-1}$, this can be rewritten as
$ \,f = (\ell_1+i\ell_2)^2+(\ell_1-i\ell_2)^2+\sum_{j=3}^n \pm \ell_j^2$.
This decomposition  represents a point in
 ${\rm VSP}(f)_\RR \backslash \overline{{\rm SSP}(f)_\RR}$.
 There is an open set of such~points.
 
 Let $f$ be definite and consider any point in
${\rm VSP}(f)_\RR$. It corresponds to a decomposition
  $$ f\,\, = \,\,\sum_{j=1}^k
  \bigl((a_{2j-1}+ib_{2j-1})(\ell_{2j-1}+i\ell_{2j})^2+(a_{2j}+ib_{2j})(\ell_{2j-1}-i\ell_{2j})^2 
  \bigr)\,\,+\sum_{j=2k+1}^n c_j \ell_j^2 ,$$
  where   $\ell_1,\ldots,\ell_n$ are
     independent real linear forms
     and the $a$'s and $b$'s are in $\RR$.
By rescaling $\ell_{2j}$ and $\ell_{2j-1}$, we obtain  $a_{2j-1}+ib_{2j-1}=1$. 
Adding the right hand side to its complex conjugate, we get $c_j \in \RR$ and
  $a_{2j}+ib_{2j}=1$. 
 The catalecticant $C(f)$ is the matrix that represents $f$.
 A change of basis shows 
that $C(f)$ has $\geq k$ negative eigenvalues, hence $k=0$.
\end{proof}

The geometry of the inclusion ${\rm SSP}(f)_\RR$ into ${\rm VSP}(f)_\RR$
is already quite subtle in the case of binary forms, i.e. $n=2$.
We call $f = a_0 x^2 + a_1 xy + a_2 y^2$
{\em hyperbolic} if its signature is $(1,1)$. Otherwise
$f$ is definite. These two cases depend on the sign of the discriminant
  $a_0 a_2 - 4 a_1^2 $.

\begin{corollary}
Let $f$ be a binary quadric of rank $2$. If $f$ is definite then ${\rm SSP}(f)_\RR={\rm VSP}(f)_\RR=\PP^1_\RR$. If $f$ is hyperbolic
then ${\rm SSP}(f)_\RR$ is an interval in the circle
${\rm VSP}(f)_\RR=\PP^1_\RR$.
\end{corollary}

\begin{proof}
The apolar ideal $f^\perp$ is generated by two quadrics $q_1,q_2$
in $\RR[x,y]_2$. Their pencil $\PP(f^\perp_2) $ is ${\rm VSP}(f) \simeq \PP^1$. A real point
$(u:v)  \in \PP^1_\RR = {\rm VSP}(f)_\RR$ may or may not be in ${\rm SSP}(f)_\RR$.
The fibers of the map  $\PP^1_\RR \rightarrow \PP^1_\RR$
given by $(q_1,q_2)$ consist of two points, corresponding to
the decompositions $f = \ell_1^2 \pm \ell_2^2$.
The fiber over $(u:v)$ consists of the roots of the  quadric $u q_2 - v q_1$.
If $f$ is definite then both roots are always real.
Otherwise the discriminant  with respect to $(x,y)$, which is a quadric in $(u,v)$,
divides $\PP^1_\RR$ into ${\rm SSP}(f)_\RR$ and its complement.
\end{proof}

\begin{example} 
\label{ex:x^2-y^2}
\rm
Fix the hyperbolic quadric $f = x^2 - y^2$. We 
take $ \, q_1 = xy$ and $q_2= x^2+y^2 $.
The quadric $u q_2 - v q_1 = u (x^2{+}y^2)- v xy $
has two real roots if and only if  
$(2u{-}v)(2u{+}v) < 0 $. 
Hence ${\rm SSP}(f)_\RR$
is the interval in $\PP^1_\RR$ defined by
$ -1/2 <  u/v < 1/2$.  In the topological
description in Theorem \ref{thm:quadrics},
the group $G$ is trivial, and ${\rm SSP}(f)_\RR$ is identified with the group
$$ {\rm SO}^{+}(1,1) \quad = \quad
\biggl\{
\begin{pmatrix}
{\rm cosh}(\alpha) & {\rm sinh}(\alpha) \\
{\rm sinh}(\alpha) & {\rm cosh}(\alpha) 
\end{pmatrix} \,\, : \,\,  \alpha \in \RR
\biggr\}.
$$
The homeomorphism between ${\rm SO}^{+}(1,1)$ and the 
interval between $-1/2$ and $1/2$ is given by
$$ \alpha \,\,\, \mapsto \,\,\,
\frac{u}{v} \,\,=\,\,
\frac{ {\rm cosh}(\alpha) \cdot {\rm sinh}(\alpha)}{{\rm cosh}(\alpha)^2 + {\rm sinh}(\alpha)^2}.
$$
The resulting factorization $\, u (x^2+y^2) - v xy \,= \,
\bigl({\rm sinh}(\alpha) x- {\rm cosh}(\alpha) y \bigr)
\bigl({\rm cosh}(\alpha) x- {\rm sinh}(\alpha) y \bigr)\,$
yields the decomposition $\, f \, = \,
\bigl({\rm cosh}(\alpha) x+ {\rm sinh}(\alpha) y\bigr)^2 \, - \,
\bigl({\rm sinh}(\alpha) x+ {\rm cosh}(\alpha) y\bigr)^2 $.
\hfill $ \diamondsuit$
\end{example}

It is instructive to examine the topology of the family of curves
${\rm SSP}(f)_\RR$ as $f$ runs over the projective
plane $\PP^2_\RR = \PP(\RR[x,y]_2)$.
This  plane is divided by an oval into two regions: 
\begin{itemize}
\item[(i)] 
the interior region $\{a_0 a_2 - 4 a_1^2 < 0 \}$ is a disk, and it parametrizes
the  definite quadrics; 
\item[(ii)] 
the exterior region $ \{a_0 a_2 - 4 a_1^2 > 0 \}$ is a M\"obius strip, 
consisting of
 hyperbolic quadrics.
\end{itemize}
Over the disk, the circles ${\rm VSP}(f)_\RR$ provide a trivial $\PP^1_\RR$--fibration.
Over the M\"obius strip, there is a twist. Namely, if we travel around 
the disk, along an  $\mathbb S^1$ in the M\"obius strip, then
the two endpoints of   ${\rm SSP}(f)_\RR$  get switched.
Hence, here we get the twisted circle bundle.

\smallskip

The topic of this paper is ternary forms, so  we now fix $n=3$.
A real ternary form of rank $3$ is either
definite or hyperbolic.
In the definite case, the normal form is
$f = x^2+y^2+z^2$, and
${\rm SSP}(f)_\RR= {\rm VSP}(f)_\RR = {\rm SO}(3)/G$,
where $G$ has order $24$.
In the hyperbolic case, the normal form is
$f = x^2+y^2-z^2$, and
$\overline{{\rm SSP}(f)_\RR} \subsetneq  {\rm VSP}(f)_\RR =  \overline{{\rm SO}^{+}(2,1)/G}$,
where $G$ has order~$4$.
These spaces are three-dimensional,
and they sit inside the complex Fano threefold $V_5$,
as seen in Table \ref{tab:VSPs of ternary forms}.
We follow  \cite{Muk, RSpolar} in developing
our algebraic approach to ${\rm SSP}(f)_\RR$.
This sets the stage for our
study of ternary forms of degree $ d \geq 4$ in the later sections.

Fix $S = \RR[x,y,z]$ and $f \in S_2$ a quadric of rank $3$.
The apolar ideal $f^\perp \subset S$ is artinian, Gorenstein,
and it has five quadratic generators.
Its minimal free resolution has the form
\begin{equation}
\label{eq:5320}
0\longrightarrow S(-5)\longrightarrow S(-3)^5
\stackrel{A}{\longrightarrow} S(-2)^5\longrightarrow S\longrightarrow 0.
\end{equation}
 By the Buchsbaum-Eisenbud Structure Theorem, we can choose bases so that 
the matrix $A$  is skew-symmetric. The entries are linear, so
 $A=xA_1+yA_2+zA_3$
where $A_1,A_2,A_3$ are real skew-symmetric $5 {\times} 5$-matrices.
More invariantly, the matrices $A_1,A_2,A_3 $ lie in $ \bigwedge^2 f^\perp_3
\simeq \RR^{10}$.
The five quadratic generators of the apolar ideal $f^\perp$ are the
$4 {\times} 4$-subpfaffians of~$A$. 

The three points $(a_i:b_i:c_i)$ in a decomposition (\ref{eq:waring}) 
are defined by three of the five quadrics. Hence ${\rm VSP}(f)$
is identified with a subvariety of the Grassmannian ${\rm Gr}(3,5)$,
defined by the condition that the three quadrics are the
minors of a $2\times 3$ matrix with linear entries. 
 Equivalently, the chosen three quadrics need to have two linear syzygies. After taking a set of
 five minimal generators of $f^\perp$ containing three  such quadrics, the matrix $A$ has the form  
\begin{equation}
\label{eq:specialA55}
A \,\, = \,\, \begin{pmatrix}  \, \,\star & T \,\, \\       -T^t & 0\,\,
    \end{pmatrix}.
    \end{equation}
Here, $0$ is the zero $2\times 2$ matrix and $T$ is a $3\times 2$ matrix
of linear forms.
The $2 \times 2$-minors of $T$ -- which are also pfaffians of $A$ -- 
are the three quadrics defining the points $(a_i:b_i:c_i)$.

\begin{proposition} \label{prop:Mukai35}
The threefold $\,{\rm VSP}(f)$ is the intersection of
the Grassmannian ${\rm Gr}(3,5)$, in its
Pl\"ucker embedding in $\,\PP(\bigwedge^3 f^\perp_3) \simeq \PP^9$, with the
$6$-dimensional linear subspace
\begin{equation}
\label{eq:threelineareqns}
\PP^6_A \,\, = \,\, \bigl\{ \,U \in \PP^9 \,\,:\,\,
U \wedge A_1= U\wedge A_2=U\wedge A_3=0 \,\bigr\}. 
\end{equation}
\end{proposition}

\begin{proof}
This fact was first observed by Mukai \cite{Muk92}.
See also \cite[\S 1.5]{RS}. If $U = u_1 \wedge u_2 \wedge u_3$
lies in this intersection then the matrix $A$ has the form 
(\ref{eq:specialA55}) for any basis that contains $u_1,u_2,u_3$.
\end{proof}

Note that any general codimension $3$ linear section of
${\rm Gr}(3,5)$ arises in this manner. In other words,
we can start with three skew-symmetric
$5 \times 5$-matrices $A_1,A_2,A_3$, and obtain
${\rm VSP}(f) = {\rm Gr}(3,5) \cap \PP^6_A$ for
a unique quadratic form $f$. In algebraic geometry,
this Fano threefold is denoted $V_5$. It has degree $5$ in $\PP^9$, and is known as the
{\em quintic del Pezzo threefold}.

Our space ${\rm SSP}(f)_\RR$ is a semialgebraic subset of
the real Fano threefold ${\rm VSP}(f)_\RR \subset \PP^9_\RR$.
If $f$ is hyperbolic then the inclusion is strict.
We now extend Example \ref{ex:x^2-y^2} to this situation.

\begin{example}
\label{ex:x^2+y^2-z^2} \rm
We shall compute the algebraic representation of
${\rm SSP}(f)_\RR$ for $f = x^2+y^2-z^2$.
The apolar ideal $f^\perp$ is generated by the
$4 \times 4$ pfaffians of the skew-symmetric matrix
\begin{equation}
\label{eq:5x5A}
A \,\,= \,\, \begin{small} \begin{pmatrix}
      0 & x &-y & z & 0 \\
     -x & 0 &-z & y &-y \\
      y & z & 0 & 0 &-x \\
     -z &-y & 0 & 0 & 0 \\
      0 & y & x & 0 & 0 
      \end{pmatrix}\end{small} \,\, = \,\,
      x (e_{12} - e_{35})
      - y (e_{13}-e_{24}+e_{25}) + z(e_{14}-e_{23}).
\end{equation}
Here $e_{ij} = e_i \wedge e_j$.
This is in the form (\ref{eq:specialA55}).
We fix affine coordinates on ${\rm Gr}(3,5)$ as follows:
\begin{equation}
\label{eq:affinecoord}
 U \,=\, {\rm rowspan} \,{\rm of} \,
\begin{pmatrix}
1 & 0 & 0 & a & b \\
0 & 1 & 0 & c & d \\
0 & 0 & 1 & e & g \\ 
\end{pmatrix}.
\end{equation}
If we write $p_{ij}$ for the signed $3 \times 3$-minors
obtained by deleting columns $i$ and $j$ from this
$3 \times 5$-matrix, then we see that 
${\rm VSP}(f) = \PP^6_A \cap {\rm Gr}(3,5)$ is defined by the
affine equations
\begin{equation}
\label{eq:V5local}
 \begin{matrix}
  p_{12} - p_{35} &= &  a d - b c \,+ \, e & = & 0, \\
  p_{13}-p_{24}+p_{25} & =  & be -a g  \, + \,d  \,+\,c  & = & 0, \\
  p_{14}-p_{23} & = & b \,+\,de - c g   & = & 0. \\
  \end{matrix}
\end{equation}
We now transform (\ref{eq:5x5A}) into the coordinate system given by
$U$ and its orthogonal complement:
\begin{equation}
\label{eq:wenowtransform}
\begin{pmatrix}   \star & T \,\, \\       -T^t & 0\,\,
\end{pmatrix}
\quad = \quad
\begin{small}
\begin{pmatrix} 1   &  0 &     0 &   a &    b \\
 0 &   1 &   0  &  c  &   d \\
0  &  0  &  1  &  e  &   g \\
a  &  c &   e  &  \!-1\!  &  0 \\
b  &  d &   g &   0  &   \!\!-1 
\end{pmatrix} \end{small} \cdot A \cdot
\begin{small}
\begin{pmatrix}
    1  &  0  &  0  &  a  &   b \\
    0  &  1  &  0  &  c  &   d \\
    0  &  0   & 1  &  e  &   g \\
    a  &  c  &  e  &  \! -1\!  &  0 \\
    b  &  d  &  g  &  0  &   \!\! -1
    \end{pmatrix} \end{small}.
\end{equation}
The lower right $2 \times 2$-block is zero precisely when (\ref{eq:V5local}) holds.    
The upper right block equals
$$ T \,= \,
\begin{pmatrix}
     (b e+c) x+(-a c+b c-e) y-(a^2+1) z &  (b g+d) x+(-a d+b d-g) y -a b z\\
(d e-a) x+(-c^2+c d-1) y-(a c+e) z &  (d g-b) x+(-c d+d^2+1) y-(b c+g) z \\
e g x+(-c e+c g+a) y-(a e-c) z &  (g^2+1) x+(-d e+d g+b) y-(b e-d) z \\
\end{pmatrix} \!.
$$
Writing $T = xT_1 + y T_2 + z T_3$, we regard $T$ as a
$2 \times 3 \times 3$ tensor with slices $T_1,T_2,T_3$
whose entries are quadratic polynomials in $a,b,c,d,e,g$.
The {\em hyperdeterminant} of that tensor equals
\begin{equation}
\label{eq:schlafli}
 \begin{matrix} \! {\rm Det}(T) & \!\! =\!\! &
{\rm discr}_w \bigl( \,{\rm Jac}_{x,y,z}( \,\,T \cdot  
\binom{1}{w}\, )\,  \bigr)   \qquad \qquad  \quad\,\,
\hbox{(by Schl\"afli's formula \cite[\S 5]{Ott})} \smallskip \\
& \! \! = \! \! & \!\!
27  a^8  c^2  d^6  g^4+54  a^8  c^2  d^4  g^6+27  a^8  c^2  d^2  g^8 + \cdots  
-4  d^2+2  e^2 -6  e  g-8  g^2-1.
\end{matrix}
\end{equation}
In general, the expected degree of the hyperdeterminant of this form is $24$. In this case, after some cancellations occur, this is a polynomial in $6$ variables of degree $20$ with $13956$ terms.
Now, consider any real point $(a,b,c,d,e,g)$ 
that satisfies (\ref{eq:V5local}). The $2 \times 2$-minors of $T$
define three points  $\ell_1,\ell_2,\ell_3$ in the complex projective plane $\PP^2$.
These three points are all real if and only if ${\rm Det}(T) < 0$.
\hfill $ \diamondsuit$
\end{example}

Our derivation establishes the following result
for the hyperbolic quadric $f  = x^2 + y^2-z^2$.
The solutions of (\ref{eq:V5local}) correspond to the
decompositions $f = \ell_1^2 + \ell_2^2 - \ell_3^2$,
as described above.

\begin{corollary} \label{cor:inaffine}
In affine coordinates  on the Grassmannian~${\rm Gr}(3,5)$, 
the real threefold ${\rm VSP}(f)_\RR$ is defined by the quadrics {\rm (\ref{eq:V5local})}.
The affine part of  ${\rm SSP}(f)_\RR \simeq {\rm SO}^{+}(2,1)/G$ 
is the semialgebraic subset of 
  points $(a,\ldots,e,g)$ at which the hyperdeterminant
${\rm Det}(T)$ is negative.
\end{corollary}

We close this section with an interpretation of
 hyperdeterminants (of next-to-boundary format) 
as {\em Hurwitz forms} \cite{Stu}. This will be used in later sections
to generalize Corollary \ref{cor:inaffine}.

\begin{proposition} \label{prop:schlafli}
The hyperdeterminant of format $m \times n \times (m{+}n{-}2)$
equals the Hurwitz form (in dual Stiefel coordinates) of the variety of
$m \times (m{+}n{-}2)$-matrices of rank $\leq m-1$.
The maximal minors of such
a matrix whose entries are linear forms in $n$ variables
define $\binom{m+n-2}{n-1}$ points in   $\PP^{n-1}$,
and the above hyperdeterminant vanishes  when
two points coincide.
\end{proposition}

\begin{proof}
Let $X$ be the variety of $m \times (m+n-2)$-matrices of rank $\leq m-1$.
By \cite[Theorem 3.10, Section 14.C]{GKZ}, the Chow form of $X$ 
equals the hyperdeterminant of boundary format $m \times n \times (m+n-1)$. 
The derivation can be extended to next-to-boundary format, and it shows that
the $m \times n \times (m+n-2)$ hyperdeterminant is the Hurwitz form of $X$.
The case $m=n=3$ is worked out in \cite[Example 4.3]{Stu}.
\end{proof}

In this paper we are concerned with the case $n=3$.
In Corollary \ref{cor:inaffine} we took $m=2$.

\begin{corollary} \label{cor:schlafli}
The hyperdeterminant of format $3 \times m \times (m+1)$
is an irreducible homogeneous polynomial
of degree $12 \binom{m+1}{3}$.
It serves as the discriminant for 
ideals of $\binom{m+1}{2}$ points in $\PP^2$.
\end{corollary}

\begin{proof}
The formula $12 \binom{m+1}{3}$ is derived from the generating function in
\cite[Theorem 14.2.4]{GKZ}, specialized to
 $3$-dimensional tensors in \cite[\S 4]{Ott}.
 For the geometry see
 \cite[Theorem 5.1]{Ott}.
\end{proof}

\section{Cubics}

The case $d=3$ was studied by Banchi \cite{Ban}.
He gave a detailed analysis of the real ranks
of ternary cubics $f \in \RR[x,y,z]_3$ with focus
on the various special cases. In this section, we
 consider a general real cubic $f$. We shall prove the 
 following result on its real decompositions.

\begin{theorem} \label{thm:diskmobius}
The semialgebraic set ${\rm SSP}(f)_\RR$
is either a disk in the real projective plane or a disjoint union of a disk  and a M\"{o}bius strip.
The two cases are characterized in
Table~\ref{tab:behavior}.
The algebraic boundary of $\,{\rm SSP}(f)_\RR$ is an
 irreducible sextic curve that has nine cusps.
\end{theorem}

Our point of departure is the 
following fact which is well-known,
e.g.~from \cite[\S 5]{Ban}  or \cite{BBO}.

\begin{proposition}\label{prop:generic}
The real rank of a general ternary cubic is $R(3) = 4$, so it agrees with the
complex rank. Hence, the closure of
$\,\mathcal{R}_3$ is all of $\RR[x,y,z]_3$, and its
boundary $\partial \mathcal{R}_3$ is empty.
\end{proposition}

\begin{proof}
Every smooth cubic curve in 
the complex projective plane $\PP^2$ can be transformed,
by an invertible linear transformation $\tau \in {\rm PGL}(3,\CC)$,
into the {\em Hesse normal form} (cf.~\cite{AD}):
\begin{equation}
\label{eq:hesse} f \,\,=\,\, x^3 + y^3 + z^3 + \lambda xyz . 
\end{equation}
Suppose that the given  cubic is defined over $\RR$.
It is known classically that the matrix $\tau$ can be chosen
to have real entries. In particular,  the parameter $\lambda$ will be real.
Also,   $\lambda \not= -3$; otherwise
the curve would be singular. Banchi \cite{Ban} observed that
$\,24 (\lambda+3)^2 f\,$ is equal to
$$ \bigl[(6 + \lambda) x - \lambda y - \lambda z \bigr]^3  \,+\, 
\bigl[ (6 + \lambda) y - \lambda x - \lambda z \bigr]^3  \, + \,
\bigl[ (6 + \lambda) z - \lambda x - \lambda y \bigr]^3  \, + \,\,
\lambda (\lambda^2 + 6 \lambda + 36) \bigl[ x+y+z \bigr]^3 .
$$
By applying $\tau^{-1}$, one obtains
the decomposition for the original cubic.
The entries of the transformation matrix $\tau \in {\rm PGL}(3,\RR)$
can be written in radicals
in the coefficients of~$f$. 
The corresponding Galois group is solvable and has order $432$. It
is the automorphism group of the Hesse pencil; see e.g.~\cite[Remark 4.2]{AD} or 
\cite[Section 2]{CS}.
\end{proof}

\begin{remark}\label{rem:Hessereal} \rm
The Hesse normal form (\ref{eq:hesse}) 
is well-suited for this real structure. 
For any fixed isomorphism class of a real elliptic
curve over $\CC$, there are two isomorphism classes over $\RR$,
by  \cite[Proposition 2.2]{SilvermanAdvanced}.
We see this by considering
the j-invariant of the Hesse curve:
\begin{equation}
\label{eq:jinv}
 j (f)\,\, = \,\,\, -{\frac {\lambda^3 \left( \lambda -6 \right)^3 \left( \lambda^2+6 \lambda +36
 \right)^3}{ \left( \lambda+3 \right)^3 \left( \lambda^2-3\lambda +9 \right)^3}} .
 \end{equation}
For a fixed real value of $j(f)$, this equation has two real solutions
$\lambda_1$ and $\lambda_2$. These two elliptic curves
are isomorphic over $\CC$ but not over $\RR$.
They are distinguished by the sign of the degree $6$ invariant $T$
of ternary cubics, which takes the following value  for the Hesse~curve:
\begin{equation}
\label{eq:Tdeg6}
 T(f) \,\,=\,\,1-\frac{4320\lambda^3+8\lambda^6}{6^6}.
 \end{equation}
If $T(f)  =0$ then the two curves
  differ by the sign of the Aronhold invariant.
  This proves that any real smooth cubic is isomorphic over $\RR$ to exactly one element of the Hesse pencil. 

An illustrative example is  the Fermat curve $   x^3+y^3+z^3$.
 It is unique over $\CC$, but it has
    two distinct real models, corresponding to $\lambda = 0$ or $6$.
 The case $\lambda=6$ is isomorphic over $\RR$ to $g = x^3+(y+iz)^3+(y-iz)^3$.
 This real cubic satisfies  $\crk(g) = 3$ but $\rrk(g) = 4$. Here,
the real surface $\,{\rm VSP}(g)_\RR\,$ is non-empty,
but its semialgebraic subset
 $\,{\rm SSP}_\RR(g)\,$ is empty.
\end{remark}

We now construct the isomorphism ${\rm VSP}(f) \simeq \PP^2$  for ternary cubics $f$ 
 as shown in Table~\ref{tab:VSPs of ternary forms}.
 The apolar ideal $f^\perp$ is  a complete intersection
generated by three quadrics $q_0,q_1,q_2$. We denote this {\it net} of quadrics by 
$f^\perp_2$. Conversely, any such 
complete intersection determines a unique cubic $f$. The linear system
$f^\perp_2$ defines a branched $4:1$ covering of projective planes:
$$ F  :\,\PP^2\rightarrow \PP^2, 
\,(x:y:z) \mapsto \bigl( q_0(x,y,z): q_1(x,y,z): q_2(x,y,z) \bigr). $$
We regard $F$ as a map from 
$\PP^2 $ to the Grassmannian $ {\rm Gr}(2,f^\perp_2)$ 
of $2$-dimensional subspaces of $f^\perp_2 \simeq \CC^3$. It
takes a point $\ell$ to the pencil of quadrics in $f^\perp_2$ 
that vanish at $\ell$. The fiber of $F$ is the
base locus of that pencil. Let $B\subset \PP^2$ be the branch
locus of $F$. This is a curve of degree six. 
The fiber of $F$ over any point in $\PP^2 \backslash B$ consists
of four points $\ell_1,\ell_2,\ell_3,\ell_4$, and these  determine decompositions
$\,f = \ell_1^3 + \ell_2^3 + \ell_3^3 + \ell_4^3$.
In this manner, the rank $4$ decompositions of $f$ are
in bijection with the points of $\PP^2\backslash B$.
We conclude that ${\rm VSP}(f) = {\rm Gr}(2,f^\perp_2) \simeq \PP^2$.

\begin{proof}[Second proof of Proposition \ref{prop:generic}]
We follow a geometric argument, due to De Paolis in
1886, as presented in \cite[\S 5]{Ban} and \cite[\S 3]{BBO}.
Let $H(f) $ be the Hessian of $f$, i.e.~the $3 \times 3$ determinant of
second partial derivatives of $f$.
We choose a real line 
$\ell_1$ that intersects the cubic $H(f)$ in three distinct real points.
The line $\ell_1$ is identified with its defining linear form, and hence
with a point in the dual $\PP^2$. That $\PP^2$ is the domain of $F$.
We may assume that $F(\ell_1)$
is not in the branch locus $ B$. There exists a decomposition
$\,f = \ell_1^3 + \ell_2^3 + \ell_3^3 + \ell_4^3$,
where $\ell_2,\ell_3,\ell_4 \in \CC[x,y,z]_1$.
We claim that the $\ell_i$ have real coefficients.
Let $\partial_p(f)$ be the    polar conic of  $f$ with respect to
 $p  = \ell_1 \cap \ell_2$. This conic is  a $\CC$-linear
   combination of $\ell_3^2$ and $\ell_4^2$. It is singular at 
the point   $\ell_3 \cap \ell_4$.
In particular, $p$ belongs to $\ell_1$ and to the cubic $ H(f)$. Hence, 
$p$ is a real point, the conic  $\partial_p (f)$ is real, and
its singular point $\ell_3 \cap \ell_4$ is real.
The latter point is distinct from $p = \ell_1 \cap \ell_2$ because $f$ is smooth.
After relabeling, all pairwise intersection
points of the lines $\ell_1,\ell_2,\ell_3,\ell_4$ are distinct and real.
Hence the lines themselves are real.
\end{proof}

The key step in the second proof is the choice of
 the  line $\ell_1$. In practise,
 this is done by sampling
linear forms $\ell_1 $ from $ \RR[x,y,z]_1$ until
 $H(f) \cap \ell_1$  consists of three real points $p$.
 For each of these, we compute
the singular point of the conic $\partial_p (f)$
and connect it to $p$ by a line.
This gives the lines
$\ell_2,\ell_3,\ell_4 \in \RR[x,y,z]_1$. 
The advantage of this method is that the
coordinates of the $\ell_i$ live in a cubic
extension, and are easy to express in terms of radicals.

In order to choose the initial line $\ell_1$ more
systematically, we must
understand the structure of ${\rm SSP}(f)_\RR$.
This is our next topic.
By definition, ${\rm SSP}(f)_\RR$ is the
locus of real points $p \in \PP^2 = {\rm Gr}(2,f^\perp_2)$ for which the
fiber $F^{-1}(p)$ is fully real.
Such points $p$ have the form $p = F(\ell)$
where $\ell$ is a line that meets the Hessian
cubic $H(f)$ in three distinct real points.

\begin{example} \label{Explicit example of cubic}  \rm
Let $f$ be the Hesse cubic (\ref{eq:hesse}).
The net $f^\perp_2 $ is spanned by the three quadrics 
$$ 
q_0 = \lambda x^2 - 6 yz,\,\,\,
q_1 = \lambda y^2 - 6 xz,\,\,\,{\rm and} \,\,
q_2 = \lambda z^2 - 6 xy. $$
These quadrics define the map $F: \PP^2 \rightarrow \PP^2$.
We use coordinates $(x:y:z)$ on the domain $\PP^2$
and coordinates $(a:b:c)$ on the image $\PP^2$.
The branch locus $B$ of $F$ is the sextic curve
$$  \begin{matrix}
177147 \lambda^4 (a^6+b^6+c^6) -
(1458 \lambda^8-157464 \lambda^5+4251528 \lambda^2) (a^4 b c+
ab^4c+abc^4) \\
        +(36 \lambda^{10}-5832 \lambda^7-39366 \lambda^4
        -5668704 \lambda) (a^3b^3+a^3c^3+b^3c^3) \\
          -(\lambda^{12}-216\lambda^9-61236 \lambda^6+3621672 \lambda^3+8503056)a^2b^2c^2.
\end{matrix}
$$
We regard the Hessian $H(f)$ as a curve in the image $\PP^2$.
This cubic curve equals
\begin{equation}
\label{eq:hessian}
H(f)\,\, = \,\, a^3+b^3+c^3-\frac{\lambda^3+108}{3\lambda^2}abc.
\end{equation}
The ramification locus of the map $F$ is the
 Jacobian of the net of quadrics:
\begin{equation}
\label{eq:jacobian}
C(f) \,\, = \,\, {\rm det} \begin{pmatrix}
 \frac{\partial q_0}{\partial x} &   \frac{\partial q_0}{\partial y} & \frac{\partial q_0}{\partial z} \\
 \frac{\partial q_1}{\partial x} &   \frac{\partial q_1}{\partial y} & \frac{\partial q_1}{\partial z} \\
 \frac{\partial q_2}{\partial x} &   \frac{\partial q_2}{\partial y} & \frac{\partial q_2}{\partial z} 
\end{pmatrix}
\,\, = \,\,
x^3 + y^3+z^3+ \frac{54-\lambda^3}{9\lambda}xyz. 
\end{equation}
This cubic is known classically as the {\em Cayleyan} of $f$;
see \cite[Prop.~3.3]{AD} and \cite[eqn.~(3.27)]{Dol}.

We note that the dual of the cubic $C(f)$ is the sextic $B$. 
The preimage of $B = C(f)^\vee$ under $F$ is a non-reduced curve
of degree $12$. It has multiplicity $2$ on the Cayleyan $C(f)$.
The other component is the sextic curve 
 dual to the Hessian $H(f)$.
That sextic equals
$$ \begin{matrix}
H(f)^\vee \, = \,
- 2187\lambda^8 (x^6+y^6+z^6)+(162\lambda^{10}+34992\lambda^7+1889568\lambda^4)(x^4yz+xy^4z+xyz^4)\\
+(-12\lambda^{11}+486\lambda^8-41990\lambda^5-15116544\lambda^2)(x^3y^3+x^3z^3+y^3z^3)+\\
+(\lambda^{12}-2484\lambda^9-244944\lambda^6+5038848\lambda^3+136048896)x^2y^2z^2. 
\end{matrix}
$$
So, we constructed four curves:
the cubic $C(f)$ and the sextic $H(f)^\vee$  in the domain $\PP^2 = \{(x \!:\! y \!:\! z)\}$,
and  the cubic $H(f)$ and the sextic $B = C(f)^\vee$ in the image
 $\PP^2 = \{(a \!: \! b \! : \! c)\}$.
\hfill $\diamondsuit$
\end{example}

A smooth cubic $f$ in the real projective plane
is either a connected curve, namely a pseudoline,
or it has two connected components, namely a pseudoline
and an oval. In the latter case,  $f$ is {\em hyperbolic}.
The cubic in the Hesse pencil (\ref{eq:hesse})
is singular for $\lambda = -3$, it
is hyperbolic if $\lambda <-3$, and it is not 
hyperbolic if $\lambda > -3$. This trichotomy is the key
for understanding  ${\rm SSP}(f)_\RR$. 
However, we must consider this trichotomy also for the Hessian cubic $H(f)$
in (\ref{eq:hessian}) and for the Cayleyan cubic $C(f)$ in (\ref{eq:jacobian}).
The issue is whether
their Hesse parameters
$-\frac{\lambda^3+108}{3\lambda^2}$ and $\frac{54-\lambda^3}{9\lambda}$
are bigger or smaller than the special value $-3$.
The values at which the behavior changes 
are $\lambda = -3, 0, 6$. Table \ref{tab:behavior}
summarizes the four possibilities.

Three possible hyperbolicity behaviors are
exhibited by the three cubics $f,H(f),C(f)$. One of
these behaviors leads to two different types, 
seen in the second and fourth column in Table \ref{tab:behavior}.
These two types are distinguished by the fibers
of the map $F: \PP^2 \rightarrow \PP^2$.
These fibers are classified by the 
connected components in the complement of the Cayleyan $C(f)$.
There are three such components if $C(f)$ is hyperbolic and two otherwise.
The fifth row in Table~\ref{tab:behavior} shows the
number of real points over these components.
For $6 < \lambda$, there are no real points over one component;
here, the general fibers have $4,2$ or $0$ real points.
However, for $-3 < \lambda < 0$, all fibers contain real points;
here, the general fibers have $4,2$ or $4$ real points.

\vspace{-0.4in}

\begin{center}
\begin{table}
\begin{tabular}{| l | l | l | l | l | r |}
  \hline
     & $\lambda < -3$ & $-3 <\lambda < 0$ & $0 < \lambda < 6$ &    $ 6 < \lambda$ \\
     \hline 
     $f$  & hyperbolic & not hyperbolic & not hyberbolic & not hyperbolic \\
  \hline
   $H(f)$ & not hyperbolic $\!\!\!\!$ & hyperbolic & hyperbolic & hyperbolic \\ 
  \hline
  $C(f)$  & hyperbolic  & hyperbolic & not hyperbolic & hyperbolic \\
  \hline
$\#  F^{-1}(\bullet)_\RR$ & $4,2,0$ & $4,2,4$ & $4,2$ & $4,2,0$  \\
  \hline
  ${\rm SSP}(f)_\RR$ & disk & disk $\sqcup$ M\"obius strip
   & disk & disk \\
  \hline
  $\!\!\!\!\!$
  \begin{small}
 Oriented Matroid 
 \end{small}
  $\!\!\!\!\!$
  &$(+,+,+,+)$&$\!(+,+,+,+)$ $\sqcup$ $(+,+,-,-)\!\! $ & $(+,+,+,-)$&$(+,+,+,+)$  \\
\hline
\end{tabular}
\caption{Four possible types of a real cubic $f$ of form (\ref{eq:hesse}) and its
 quadratic map $F:\PP^2 \rightarrow \PP^2$.
 \label{tab:behavior}}
\end{table}
\end{center}

\begin{figure}
    \centering
    \begin{subfigure}[b]{0.3\textwidth}
        \includegraphics[scale=0.3]{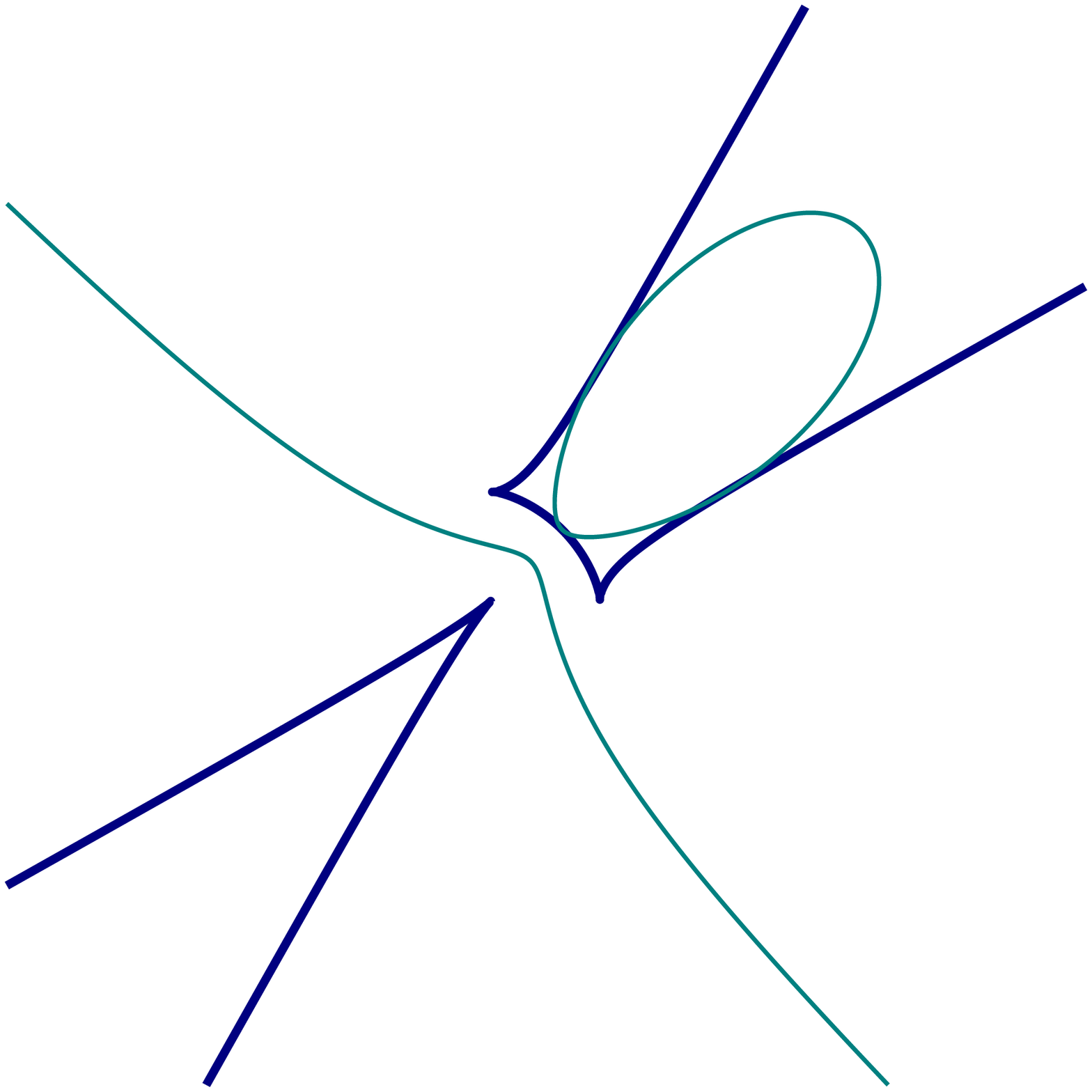}
        \caption{Domain  $\PP^2 = \{(x:y:z)\}$}
        \label{fig: Domain}
    \end{subfigure}
    ~ 
      \hskip 3cm
    \begin{subfigure}[b]{0.3\textwidth}
        \includegraphics[scale=0.3]{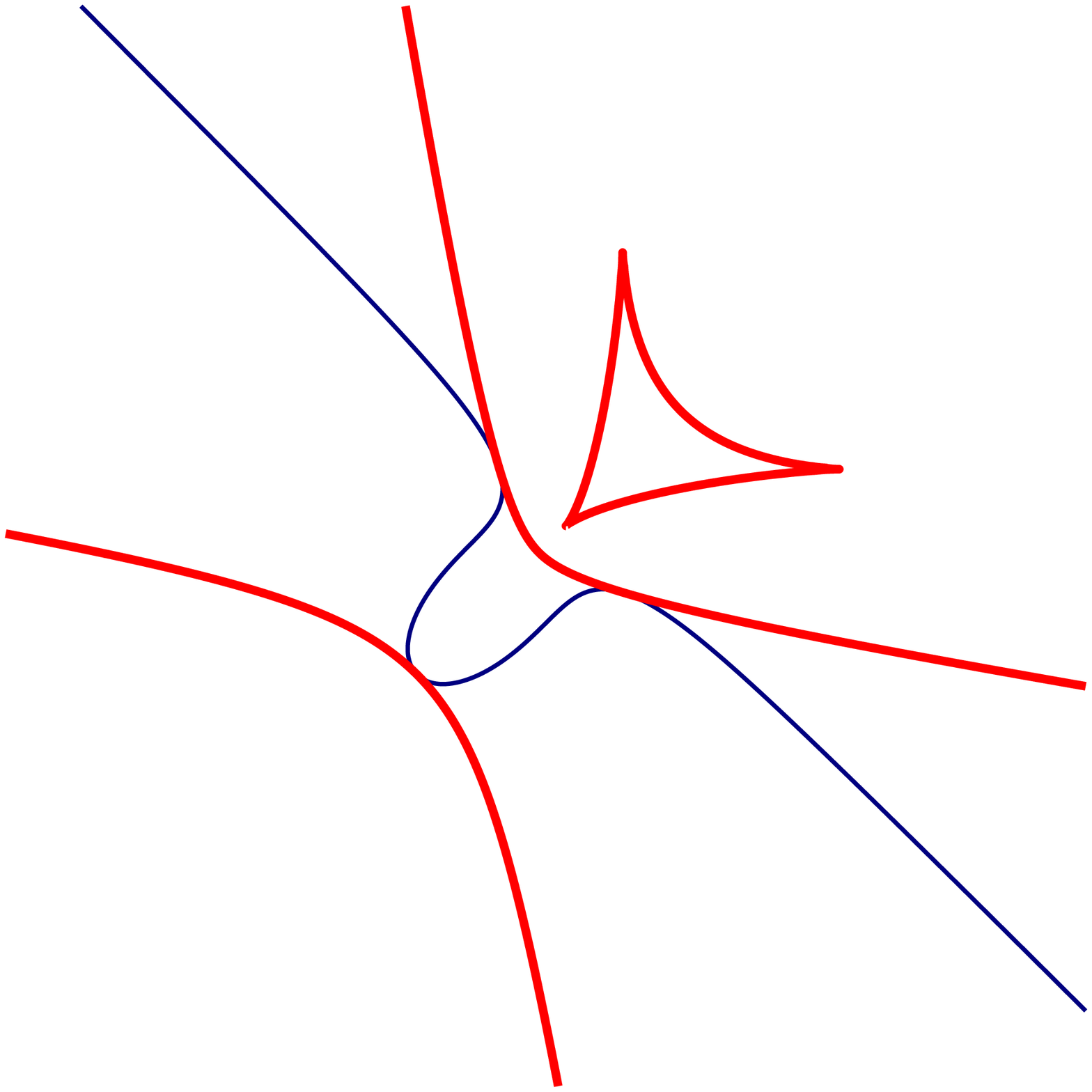}
        \caption{Image $\PP^2 = \{(a:b:c)\}$}
        \label{fig: Codomain}
    \end{subfigure}
    \vspace{-0.1in}
    \caption{\label{fig:case1} Ramification and branching for  $\lambda < -3$. 
The triangular region in (b) is ${\rm SSP}(f)_\RR$.}
\end{figure}

\begin{figure}
\centering
    \begin{subfigure}[b]{0.3\textwidth}
        \includegraphics[scale=0.3]{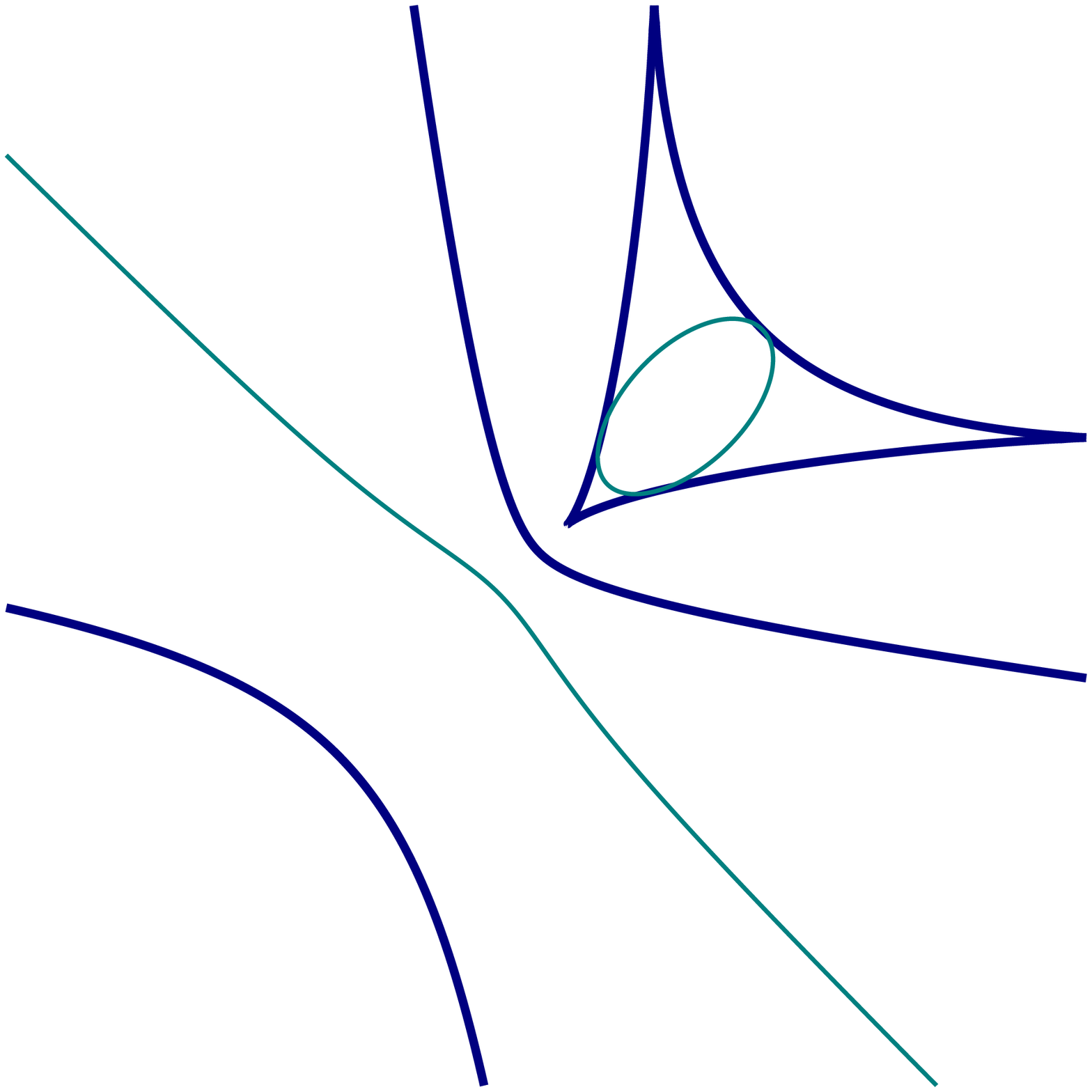}
        \label{fig: Domain}
    \end{subfigure}
    ~ 
      \hskip 3cm
    \begin{subfigure}[b]{0.3\textwidth}
        \includegraphics[scale=0.3]{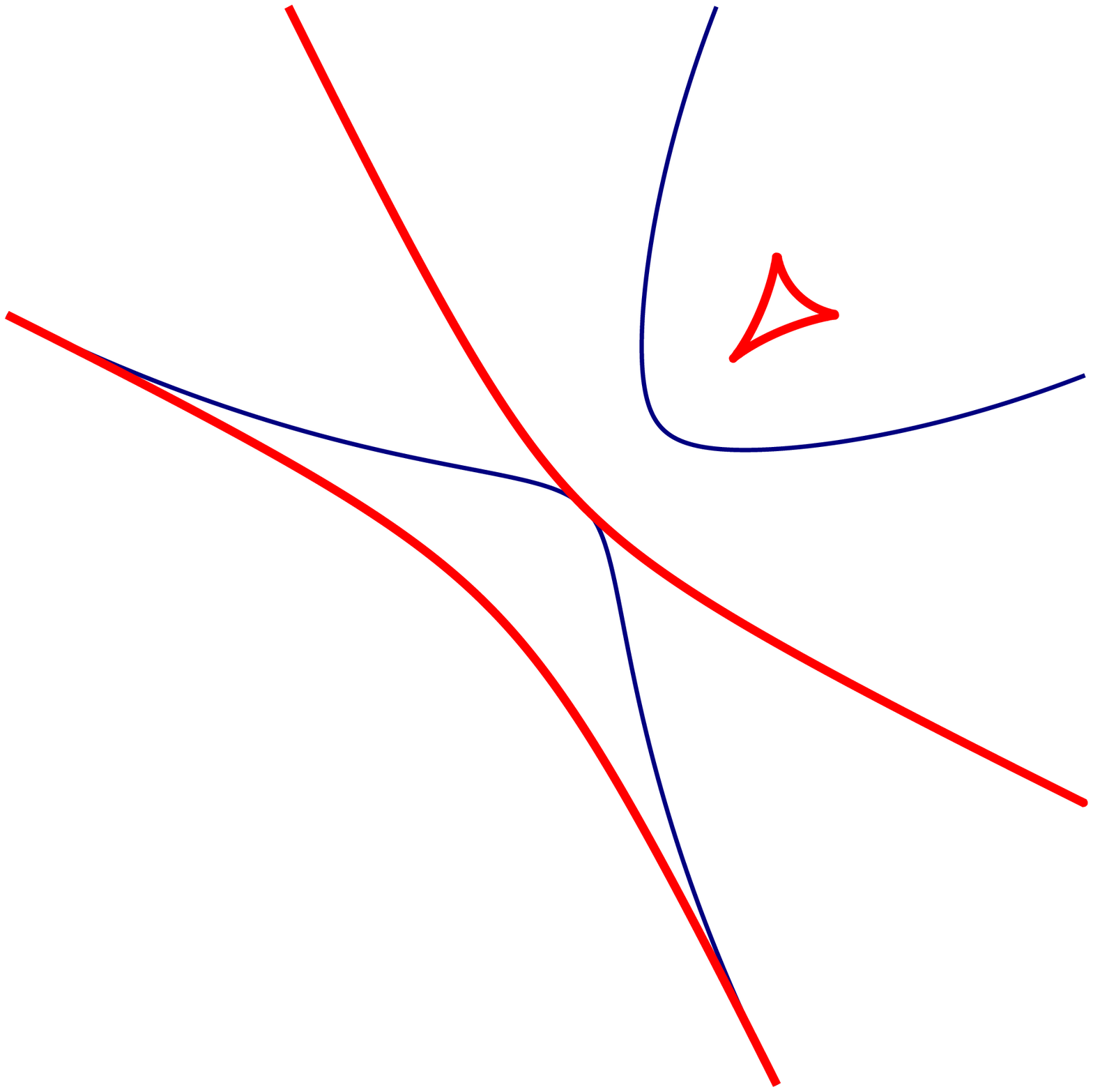}
         \label{fig: Codomain}
    \end{subfigure}
    \vspace{-0.3in}
 \caption{\label{fig:case2} Ramification and branching for  $-3 < \lambda < 0$. 
 The locus ${\rm SSP}(f)_\RR$ is bounded by the (red) sextic curve on
 the right.
 It consists of   the triangular disk  and the M\"obius strip.}
\end{figure}

\begin{proof}[Proof of Theorem \ref{thm:diskmobius}]
After a coordinate change  by a matrix $\tau \in {\rm PGL}(3,\RR)$,
we~can assume that the cubic $f$ is in the Hesse pencil (\ref{eq:hesse}).
Hence so are the associated cubics $H(f)$ and $C(f)$.
If we change the parameter $\lambda$
so that all three cubics remain smooth, 
then the real topology of the map $F$ is unchanged.
This gives four different types for ${\rm SSP}(f)_\RR$,
the locus of fully real fibers.
The sextic $B$ divides the  real projective plane
into  two or three connected components,
depending on whether its dual cubic $C(f) = B^\vee $ is hyperbolic or~not.

Figures~\ref{fig:case1}, \ref{fig:case2}, \ref{fig:case4} and \ref{fig:case3}
illustrate the behavior of the map $F$ in the four cases
given by the columns in Table \ref{tab:behavior}.
Each figure shows the plane $\PP^2$ with coordinates
$(x:y:z)$ on the left and the  plane $\PP^2$ with coordinates $(a:b:c)$
on the right. 
The map $F$ takes the left plane onto the right plane.
 The two planes are dual to each other.
In particular, points on the left correspond to lines on the right.
Each of the eight drawings shows a cubic curve and a sextic curve.
The two curves on the left are dual to the two curves on the right.

In each right diagram, the thick red curve is the branch locus $B$
and the thin blue curve is the Hessian $H(f)$.
In each left diagram, the turquoise curve is the Cayleyan
$C(f) = B^\vee$, and the thick blue curve is the
sextic $H(f)^\vee$ dual to the Hessian.
Each of the eight cubics has either two or one
connected components, depending on whether the
curve is hyperbolic or not. The complement of
the cubic in $\PP^2_\RR$ has
three or two connected components.
The diagrams verify the
hyperbolicity behavior stated in the third and fourth row
of Table~\ref{tab:behavior}.
Note that each sextic curve has the same
number of components in $\PP^2_\RR$ as its dual cubic.

\begin{figure}
    \centering
    \begin{subfigure}[b]{0.3\textwidth}
        \includegraphics[scale=0.3]{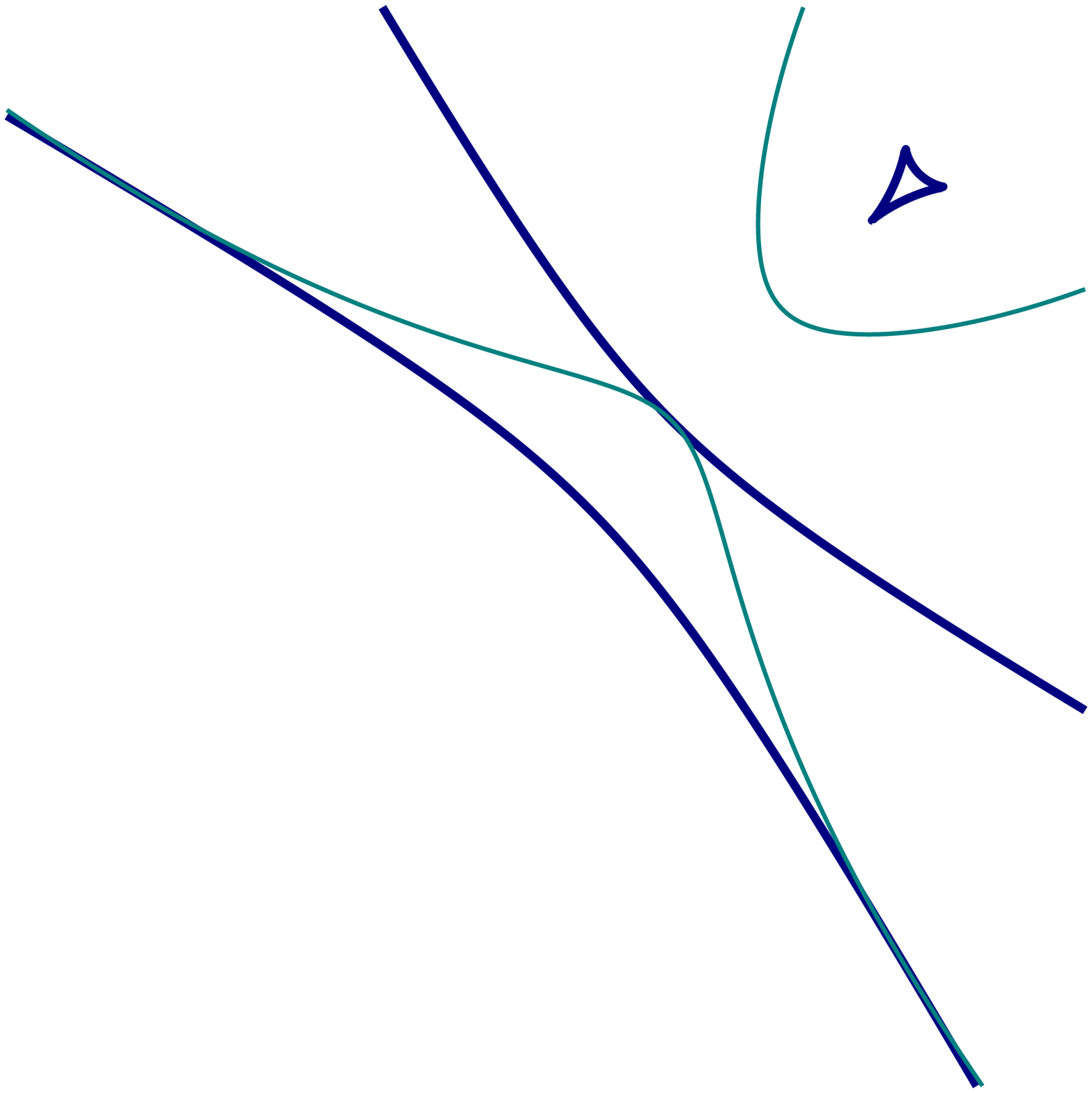}
        \label{fig: Domain}
    \end{subfigure}
    ~ 
      \hskip 3cm
    \begin{subfigure}[b]{0.3\textwidth}
        \includegraphics[scale=0.3]{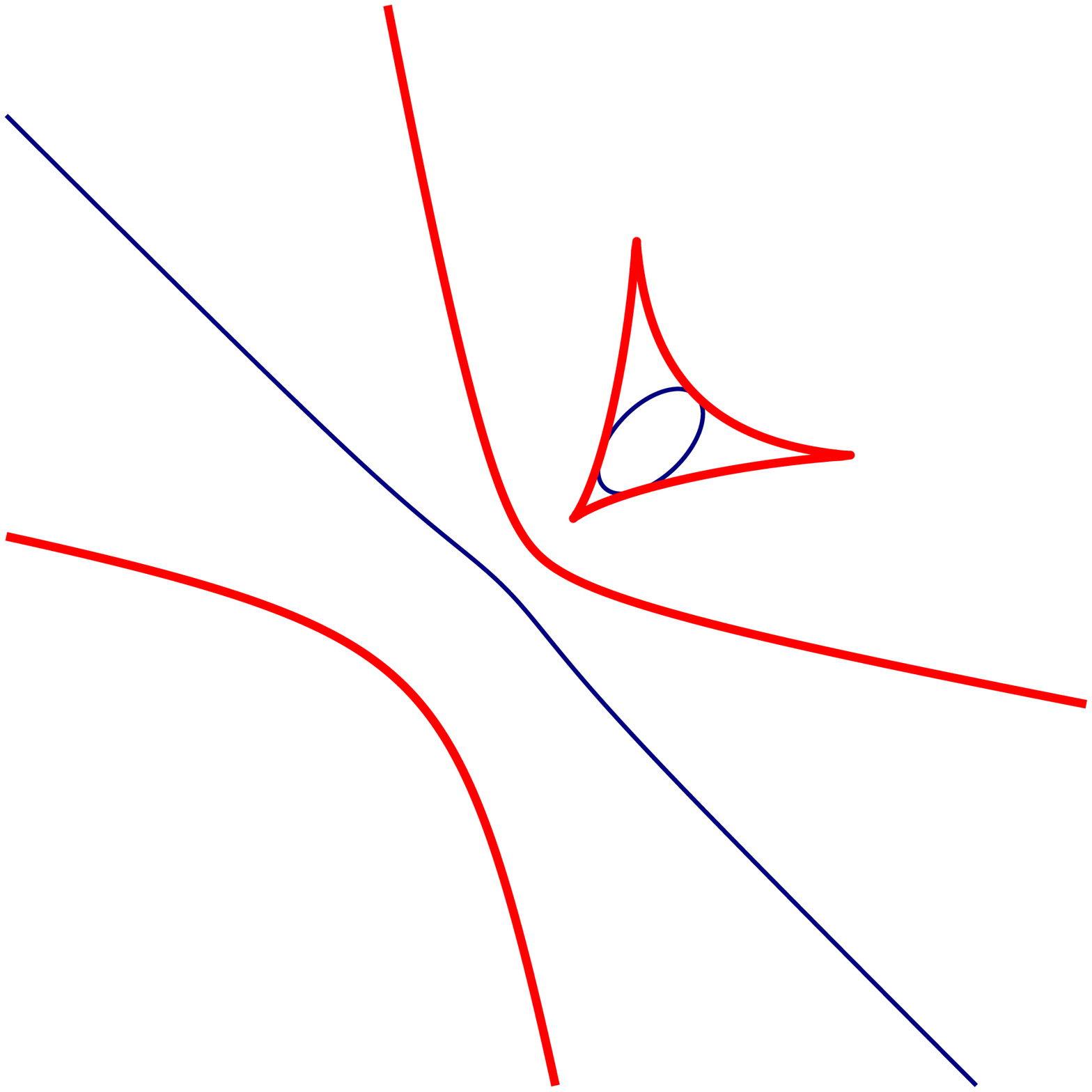}
     \end{subfigure}
     \vspace{-0.1in}
\caption{\label{fig:case4} Ramification and branching for  $\lambda > 6$. 
The triangular region is ${\rm SSP}(f)_\RR$.}
\end{figure}

Consider the three cases where $C(f)$ is hyperbolic.
These are in Figures~\ref{fig:case1}, \ref{fig:case2} and \ref{fig:case4}.
Here, $\PP^2_\RR \backslash B$ has three connected components.
The fibers of $F$ could have $0$, $2$ or $4$ real points
on these three regions.  The innermost region has  four real points
in its fibers. It is bounded by the triangular connected component of
the (red) branch curve $B$, which is dual to the pseudoline of $C(f)$.
This innermost region is connected and contractible: it is a disk in~$\PP^2_\RR$.

 If $\lambda\not\in[-3,0]$ then this disk is exactly our set  ${\rm SSP}(f)_\RR$.
  This happens in Figures \ref{fig:case1} and~\ref{fig:case4}.
 However, the case $\lambda\in(-3,0)$ is different.
This case is depicted in Figure~\ref{fig:case2}. Here,
we see that   ${\rm SSP}(f)_\RR$ 
consists of two regions. First, there is the disk as before, and
second, we have the outermost region. This region is bounded by the
oval that is shown as two unbouded branches
on the right in Figure~\ref{fig:case2}.
That region is  homeomorphic to a M\"obius strip in $\PP^2_\RR$.
The key observation is that the fibers of $F$ over that 
M\"obius strip consist of four real points.
 
\begin{figure}[t]
\vspace{-0.1in}
    \centering
    \begin{subfigure}[b]{0.3\textwidth}
        \includegraphics[scale=0.3]{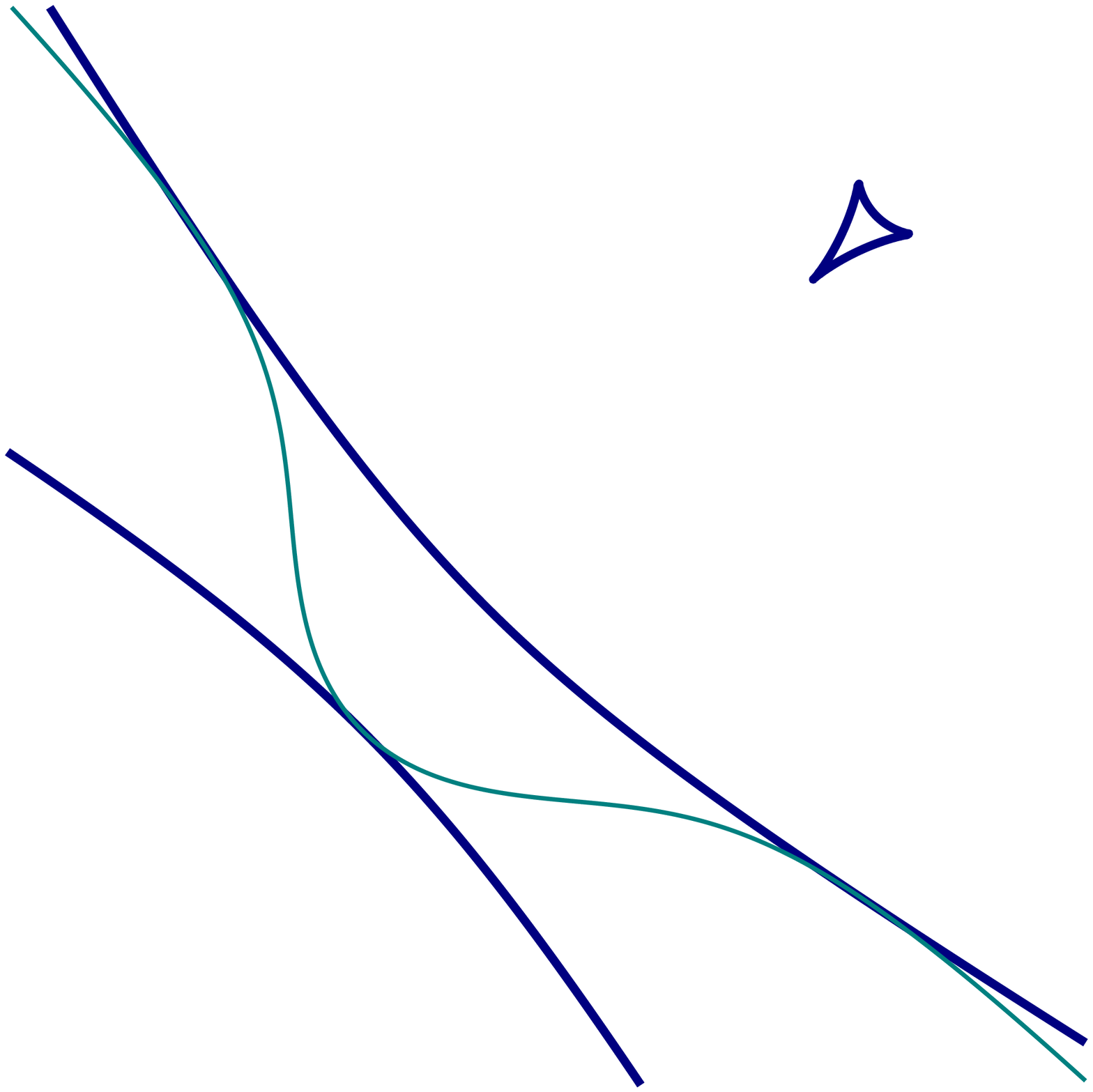}
    \end{subfigure}
    ~ 
      \hskip 3cm
    \begin{subfigure}[b]{0.3\textwidth}
        \includegraphics[scale=0.3]{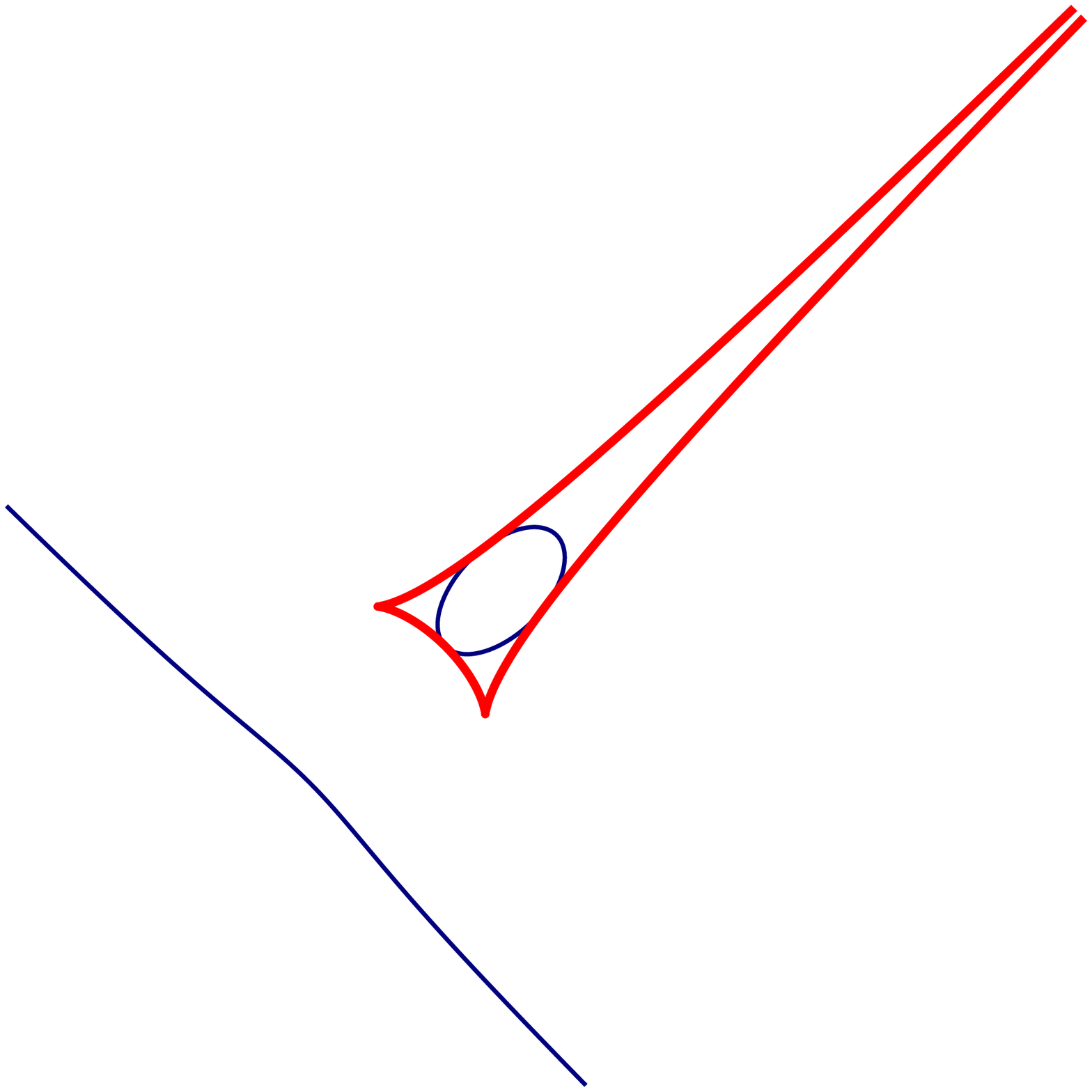}
        \label{fig: Codomain}
    \end{subfigure}
    \vspace{-0.15in}
 \caption{\label{fig:case3} Ramification and branching for  $0 < \lambda < 6$. 
The triangular region is ${\rm SSP}(f)_\RR$.}
\end{figure}

Figure~\ref{fig:case2} reveals something interesting 
for the  decompositions $f=\sum_{i=1}^4 \ell_i^3$.
These come in two different types, for $\lambda \in (-3,0)$,
 one for each of  the two connected components of
  ${\rm SSP}(f)_\RR$. Over the disk,
all four lines $\ell_i$ intersect the Hessian $H(f)$ only in its pseudoline.
Over the M\"obius strip, the
$\ell_i$ intersect the oval of $H(f)$ in two points 
and the third intersection point is on the pseudoline.
Compare this with Figure~\ref{fig:case4}: 
 the Hessian $H(f)$ is also hyperbolic, but all decompositions are of the same type: three lines 
 $\ell_i$ intersect $H(f)$ in two points of its oval and one point
 of its pseudoline, while the fourth line intersects $H(f)$ only in its pseudoline.
 
 It remains to consider the case when $C(f)$ is not hyperbolic.
 This is shown in Figure~\ref{fig:case3}.  The branch curve  $B = C(f)^\vee$ divides $\PP^2_\RR$
 into two regions, one disk and one M\"obius strip.
 The former corresponds to fibers with four real points, and the latter
corresponds to fibers with two real points. 
We conclude that ${\rm SSP}(f)_\RR$ is a disk also in this last case.
We might note, as a corollary, that
 all fibers of  $F: \PP^2 \rightarrow \PP^2 $ contain real points,
provided $0 < \lambda < 6 $.

For all four columns of Table \ref{tab:behavior}, the algebraic boundary of the 
set ${\rm SSP}(f)_\RR$ is the branch curve
$B$. This is a sextic with nine cusps because it is
dual to the smooth cubic $C(f)$.
\end{proof}

One may ask for the topological structure of the $4:1$ covering 
over ${\rm SSP}(f)_\RR$.  Over the disk, our map $F$ is $4:1$.
It maps four disjoint disks. 
Each linear form in the corresponding decompositions $f = \sum_{i=1}^4 \ell_i^3$
 comes from one of the four regions seen in the left pictures:
 
 \begin{enumerate} 
 \item[(i)] in Figure \ref{fig:case1}, inside the region bounded by $H(f)^\vee$ 
 and cut into four by  $C(f)$;
 
 \item[(ii)] in Figure \ref{fig:case2}, inside the spiky triangle bounded by $H(f)^\vee$ and cut into four by $C(f)$;
 
 \item[(iii)] in Figures \ref{fig:case4} and \ref{fig:case3}, one inside the triangle bounded by $H(f)^\vee$ and the others in the region bounded by the other component of $H(f)^\vee$ cut into three regions by $C(f)$.
 
 \end{enumerate}
The situation is even more interesting over the M\"obius strip. We can continuously change the set $\{\ell_1,\ell_2,\ell_3,\ell_4\}$, reaching in the end the same as at the beginning, but cyclicly permuted.

 

\begin{remark}  \rm Given a  ternary cubic $f$
with rational coefficients, how to 
 decide whether ${\rm SSP}(f)_\RR$ has
one or two connected components? The classification
 in  Table~\ref{tab:behavior} can  be used for this task 
as follows. We first compute the j-invariant of $f$ and then we
substitute the rational number $j(f)$ into (\ref{eq:jinv}).
 This gives a polynomial of degree $4$ in the unknown $\lambda$. 
 That polynomial has  two distinct real roots $\lambda_1 < \lambda_2$, provided
  $\,j(f) \not\in \{0,1728\}$. They satisfy 
$\lambda_1 < \lambda_2 < -3$,  or $0 < \lambda_1 < \lambda_2 < 6$,
or ($-3 < \lambda_1  < 0$ and $ 6 < \lambda_2$).
Consider the involution that swaps $\lambda_1 $ with $\lambda_2$.
This fixes the case in Figure \ref{fig:case1}, and the case in Figure \ref{fig:case3},
but it swaps the cases in Figures \ref{fig:case2} and \ref{fig:case4}.
Thus this involution preserves the hyperbolicity bahavior. We get two connected 
components, namely both the disk and M\"obius strip, only in the last case. The correct $\lambda$ 
is identified by comparing the sign of the degree six invariant $T(f)$, as in
(\ref{eq:Tdeg6}).
\end{remark}

\begin{example} \rm
The following cubic is featured in 
the statistics context of \cite[Example 1.1]{SU}:
$$ f \, = \, {\rm det} \begin{small} \begin{pmatrix}
x{+}y{+}z \! & x & y \\
\! x & x{+}y{+}z \! & z \\ 
y & z & \! x{+}y{+}z \end{pmatrix} \end{small}
\,\, = \,\,
\frac{4}{3} (x+y+z)^3 \,-\,
\frac{2}{3} (x+y)^3 \,-\,
\frac{2}{3} (x+z)^3 \,-\,
\frac{2}{3} (y+z)^3 .
$$
Its j-invariant equals $ j(f) =  16384/5$. The corresponding real Hesse 
curves have parameters $\lambda_1 = -13.506...$ and $\lambda_2 =  -5.57...$,
so we are in the case of Figure~\ref{fig:case1}.
Indeed, the curve $V(f)$ is hyperbolic, as seen in \cite[Figure 1]{SU}.
Hence ${\rm SSP}(f)_\RR$ is a disk, shaped like a spiky triangle.
The real
decomposition above is right in its center. Moreover, we can check that $T(f)<0$.
Hence $\lambda_1$ provides the unique curve in the Hesse pencil 
that is isomorphic to $f$ over $\RR$.
\hfill $\diamondsuit$
\end{example}

\begin{remark} 
\rm
The value $1728$ for the j-invariant plays a  special role.
A real cubic $f$ is hyperbolic if $j(f) > 1728$,
and it is not hyperbolic if $j(f) < 1728$. Applying
this criterion to a given cubic along with its Hessian and  Cayleyan
is useful for the classification in Table~\ref{tab:behavior}.

What happens for $j(f) = 1728$?   Here, the two real forms of the
complex curve $V(f)$ differ: one is hyperbolic and the other one is not.
For example, $f_1 = x^3- x z^2 - y^2 z $ is hyperbolic and
$f_2 = x^3 + x z^2 - y^2 z$ is not hyperbolic. These
two cubics are isomorphic over $\CC$, with $j(f_1)  = j(f_2) = 1728$,
and they are also isomorphic to their Hessians and Cayleyans.

We find noteworthy that the topology of ${\rm SSP}(f)_\RR$ 
can distinguish between the two real forms of an elliptic curve.
This happens when $j(f) < 1728< {\rm min}\bigl\{ j(H(f)), j(C(f)) \bigr\}$.
Here the two real forms of the curve correspond to 
the second and fourth column in Table \ref{tab:behavior}.
\end{remark}

We close this section by explaining the last row of Table \ref{tab:behavior}.
It concerns the {\em oriented matroid} \cite{OrMat}
of the configuration $\{(a_i,b_i,c_i) : i=1,\ldots,r\}$ 
in the decompositions (\ref{eq:waring}).
For $d=3$ the underlying matroid is always {\em uniform}.
This is the content of the following lemma.

\begin{lemma}\label{lem:nondegcubic}
Consider a ternary cubic $f=\sum_{i=1}^4\ell_i^3$ whose
apolar ideal $f^{\perp}$ is generated by three quadrics.
Then any three of the linear forms $\ell_1,\ell_2,\ell_3,\ell_4$ are linearly independent.
\end{lemma}

\begin{proof}
Suppose $\ell_1,\ell_2, \ell_3$ are linearly dependent. They 
are annihilated by a linear operator $p$ as in
(\ref{eq:keyplayer}). Let $q_1$ and $q_2$ be independent linear operators 
 that annihilate $\ell_4$. Then $p q_1$ and $p q_2$ are
 independent quadratic operators annihilating $f$. Adding a third quadric would not lead to
  a complete intersection. This is a contradiction, since $f^{\perp}$ is a complete intersection. 
\end{proof}

In the situation of Lemma \ref{lem:nondegcubic}
there is unique vector
$v = (v_1,v_2,v_3,v_4) \in (\RR \backslash \{0\})^4$ satisfying
$v_1 = 1$ and $\sum_{i=1}^4 v_i \ell_i = 0$.
The oriented matroid of $(\ell_1,\ell_2,\ell_3,\ell_4)$
is the sign vector $\bigl(+,{\rm sign}(v_2),{\rm sign}(v_3), {\rm sign}(v_4)\bigr) \in \{-,+\}^4$.
Up to relabeling there are only three possibilities:  

$ \quad (+,+,+,+)$: \ the four vectors $\ell_i$ contain the origin in their convex hull;

 $ \quad (+,+,+,-)$: \ the triangular cone spanned by $\ell_1,\ell_2,\ell_3$ in $\RR^3$ contains $\ell_4$;

 $ \quad (+,+,-,-)$: \ the cone spanned by $\ell_1,\ell_2,\ell_3,\ell_4$ is the
cone over a quadrilateral. 

For a general cubic $f$, every point in ${\rm SSP}(f)_\RR$
is mapped to one of the three sign vectors above. By continuity,
this map is constant on each connected component of ${\rm SSP}(f)_\RR$.
The last row in Table \ref{tab:behavior} shows the resulting map from the five
connected components to the three oriented matroids. 
Two of the fibers have cardinality one.
 For instance, the fiber over $(+,+,-,-)$ is the M\"obius strip in
 ${\rm SSP}(f)_\RR$. This is the first of the following two cases.
 
\begin{corollary}
For a general ternary cubic $f$, we have the following equivalences:
\begin{enumerate}
\item[(i)] 
The space
${\rm SSP}(f)_\RR$ is disconnected if and only if $f$ is isomorphic over $\RR$ to a cubic
of the form $x^3+y^3+z^3+(ax+by-cz)^3$ 
where $a,b,c$ are positive real numbers.
\item[(ii)] The Hessian
$H(f)$ is hyperbolic and the Cayleyan $C(f)$ is not hyperbolic if and only if $f$ is isomorphic 
 to $x^3+y^3+z^3
+(a x+ by+cz)^3$ 
where $a,b,c$ are positive real numbers.
\end{enumerate} 
\end{corollary}

\begin{proof}
The sign patters $(+,+,-,-)$ and $ (+,+,+,-)$ occur
in the   second and third 
column in Table \ref{tab:behavior} respectively.
The corollary is a reformulation of that fact.
The sign pattern $(+,+,+,+)$ in columns
1,2 and 4 corresponds to cubics
$x^3+y^3+z^3 -(a x+ by+cz)^3$.
\end{proof}

\begin{remark} \rm
The fiber over the oriented matroid $(+,+,+,+)$ 
consists of cubics of the form $x^3+y^3+z^3-(a x+ by+cz)^3$,
where $a,b,c  > 0$.
It may seem surprising that this space has
three components in Table \ref{tab:behavior}.
 One can pass from one component to another,
for instance, by passing through singular cubics, like
$\,24xyz =  (x+y+z)^3+(x-y-z)^3+(-x+y-z)^3+(-x-y+z)^3 $.
\end{remark}

\section{Quartics}

In this section we fix $d=4$ and we consider a general ternary quartic
$ f \in \RR[x,y,z]_4$. We have $r=R(4) = 6$
and ${\rm dim}({\rm VSP}(f)) = 3$, so the quartic $f$ admits a threefold of decompositions
\begin{equation}
\label{eq:admitsmany}
 f(x,y,z) \,\, = \,\, 
 \lambda_1 ( a_1 x+b_1 y+c_1z)^4 + 
  \lambda_2 ( a_2 x+b_2 y+c_2z)^4 + \
 \cdots + 
  \lambda_6 ( a_6 x+b_6 y+c_6z)^4 . 
  \end{equation}
  By the {\em signature} of $f$ we mean the signature of $C(f)$. This makes sense
     by Proposition~\ref{lem:2468}.

  According to the Hilbert-Burch Theorem,  the radical ideal $I_T$ of the point configuration  $\{ (a_i:b_i:c_i)\}_{i=1,\ldots,6}$  is generated by the $3 \times 3$-minors
  of a $4 \times 3$-matrix   $ T = T_1 x + T_2 y + T_3 z$,
  where $T_1,T_2,T_3 \in \RR^{4 \times 3}$.
  We interpret $T$ also as a $3 \times 3 \times 4$-tensor with entries in $\RR$,
    or as a $ 3 \times 3$-matrix of linear forms in $4$ variables.
  The determinant of the latter matrix defines the cubic surface in $\PP^3$
  that is the blow-up of    the projective plane $\PP^2$ at the six points.
    The apolar ideal $f^\perp $ is generated by seven cubics,
  and $I_T \subset f^\perp$ is generated by four of these. 

Mukai \cite{Muk} showed that ${\rm VSP}(f)$ is a Fano threefold
of type $V_{22}$, and a more detailed study of this threefold
was undertaken by Schreyer in \cite{Sch}.
The topology of the real points in that Fano threefold was studied by
Koll\'ar and Schreyer in \cite{KS}. Inside that real locus lives
the semialgebraic set we are interested in.
Namely, ${\rm SSP}(f)_\RR$ represents the set of radical ideals
$I_T$, arising from tensors $T \in \RR^{3 \times 3 \times 4}$,
such that $I_T \subset f^\perp$ and the variety $V(I_T)$
consists of six real points in $\PP^2$. This is equivalent to
saying that the cubic surface of $T$ has $27$ real lines.

Disregarding the condition $I_T \subset f^\perp$ for the moment,
this reality requirement defines a full-dimensional, connected 
semialgebraic region in the tensor space $\RR^{3 \times 3 \times 4}$.
The algebraic boundary of that region is defined by
the hyperdeterminant ${\rm Det}(T)$, which is an irreducible
homogeneous polynomial of degree $48$ in the $36$ entries of $T$.
Geometrically, this hyperdeterminant is the Hurwitz form 
in \cite[Example 4.3]{Stu}.
This is Proposition \ref{prop:schlafli} for $m=n=3$.

We are interested in those ternary forms $f$ whose apolar ideal $f^\perp$
contains $I_T $ for some $T$ in the region described above.
Namely, we wish to understand the semialgebraic set
$$  \mathcal{R}_4 \,\, = \,\, \bigl\{\, f \in \RR[x,y,z]_4 \,:\,
{\rm SSP}(f)_\RR \,\not= \,\emptyset \,\bigr\}. $$
The following is a step towards understanding the algebraic boundary 
$\partial_{\rm alg}( \mathcal{R}_4)$ of this set.

\begin{theorem}\label{Algebraic Boundary}
The algebraic boundary $\partial_{\rm alg}( \mathcal{R}_4)$ is
a reducible hypersurface in the $\PP^{14}$ of quartics. One of its
irreducible components has degree $51$; that component
 divides the quartics of signature $(5,1)$.
 Another irreducible component divides the region of hyperbolic quartics.
\end{theorem}

\begin{proof}
By \cite[Example 4.6]{BBO}, $\partial_{\rm alg}( \mathcal{R}_4)$ is
non-empty, so it must be a hypersurface. 
We next identify the component of degree $51$.
The {\em anti-polar} of a quartic $f$,
 featured in \cite[\S 5.1]{BBO} and  in \cite{Dolsing},
is defined by the following rank $1$ update
of the middle catalecticant:
\begin{equation}
\label{eq:antipolar1}
 \BQ(f)(a,b,c) \,\, := \,\, \,{\rm det}\bigl(C(f + \ell^4) \bigr)\, -\, {\rm det}\bigl(C(f) \bigr)
\qquad {\rm for} \quad \ell = a x + b y + c z. 
\end{equation}
Writing ``Adj'' for the adjoint matrix, 
the Matrix Determinant Lemma implies
\begin{equation}
\label{eq:antipolar2}
 \BQ(f)(a,b,c) \,\,:=\,\,
(a^2,ab,ac,b^2,bc,c^2) \cdot {\rm Adj}(C(f)) \cdot 
(a^2,ab,ac,b^2,bc,c^2)^T. 
\end{equation}
The coefficients of the anti-polar quartic $\BQ(f)$
are homogeneous polynomials of degree $5$ in the coefficients of $f$.
The discriminant of $\BQ(f)$ is a polynomial of degree
$135 = 27 \times 5$ in the coefficients of $f$.
A computation reveals that this  factors as
${\rm det} ( C(f) )^{14}$ times an irreducible
polynomial ${\rm Bdisc}(f)$ of degree $51$. 
We call $ {\rm Bdisc}(f)$ the {\em Blekherman discriminant} of $f$.

We claim that $ {\rm Bdisc}(f)$ is an irreducible component
of $\partial_{\rm alg}( \mathcal{R}_4)$.
Let $f$ be a general quartic  of signature $(5,1)$.
Then ${\rm det}(C(f)) $ is negative, and
the quartic $\BQ(f)$ is nonsingular.
 We claim that $\rrk(f) = 6$ if and only if the curve $\BQ(f)$ has a real point.
The only-if direction is proved 
 in a more general context in
 Lemma \ref{Blekhermangeneral}.
For the if-direction, we note that the anti-polar quartic curve  divides
    $\PP_\RR^2$  into regions where $\BQ(f)$
 has opposite signs. Hence, we can find $\ell$ such that ${\rm det}\bigl(C(f + \ell^4) \bigr)=0$, and 
 therefore  $\rrk(f) = 6$.
 Examples in \cite[\S 5.1]{BBO} show that
 $\rrk(f)$ can be either $6$ or $7$.
We conclude that, among quartics of signature $(5,1)$, the
boundary of $\mathcal{R}_4$ is given by the Blekherman discriminant 
${\rm Bdisc}(f)$ of degree $51$.

To prove that  $\partial_{\rm alg}( \mathcal{R}_4)$ 
  is reducible, we     consider the following pencil of quartics:
$$ \begin{matrix}  f_t \quad = \quad
(6x-4y+17z)^4+(4x-16y-5z)^4+(20x+2y-19z)^4  \qquad \qquad \qquad \qquad \\
\qquad \qquad -(15y-17z)^4-(13x+14y+9z)^4-(16x-6y-18z)^4 \qquad \qquad \qquad \qquad \\ 
+\,t\cdot(-2x^4{+}2x^3z{-}x^2y^2{+}2x^2yz{+}x^2z^2
{+}xy^3{-}xy^2z{-}2xyz^2{+}xz^3{+}y^4{+}y^3z{+}2y^2z^2{-}2yz^3{-}2z^4).
\end{matrix}
$$
At $t=0$, we obtain a quartic $f_0$ of signature $(3,3)$ that has real rank $6$.
One checks that $f_0$ is smooth and hyperbolic.
We substitute $f_t$ into the invariant of degree $51$ derived above,
and we note that the resulting univariate polynomial in $t$ has no positive real roots.
So, the ray $\{f_t\}$ given by $ t \geq 0$ does not intersect the
boundary component we already identified.

For positive parameters $t$, the discriminant of $f_t$ is nonzero, until $t $ reaches $\tau_1  = 6243.83....$.
This means that $f_t$ is smooth hyperbolic for
real parameters $t$ between $0$ and $\tau_1$.
On the other hand, the rank of the middle catalecticant $C(f_t)$ drops from
$6$ to $5$ when $t$ equals    $\tau_0 = 3103.22....$.
 Hence, for $\tau_0 < t < \tau_1$, the quartics $f_t$ are hyperbolic and of
    signature $(4,2)$. By \cite[Corollary 4.8]{BBO}, these quartics have real rank at least $7$.
    This means that the half-open interval given by $(0,\tau_0]$ crosses the
    boundary of $\mathcal{R}_4$ in a new    irreducible component.
   \end{proof}

\begin{remark} \rm
One of the starting points of this project was the question 
 whether
 $\rrk(f) \geq 7$ holds for all hyperbolic quartics $f$.
This was shown to be false in \cite[Remark 4.9]{BBO}.
The above quartic $f_0$ is an
alternative counterexample, with an
explicit rank $6$ decomposition over $\mathbb{Q}$.
\end{remark}

  We believe that, in the proof above, the crossing takes place at $\tau_0$, and that
    this newly discovered component is simply
    the determinant of the catalecticant ${\rm det}(C(f))$. But we 
    have not been able to certify this. Similar examples suggest that
    also the discriminant ${\rm disc}(f)$ itself appears in the real rank boundary.
    Based on this, we propose the following conjecture.

\begin{conjecture} \label{conj:3comp}
The real rank boundary $\partial_{\rm alg}(\mathcal{R}_4)$
for ternary quartics
is a reducible hypersurface of degree $84 = 6+27+51$  in $\PP^{14}$.
 It has three  irreducible components, namely the
determinant of the catalecticant,
the discriminant, and the Blekherman discriminant. Algebraically,
$$ \partial_{\rm alg}(\mathcal{R}_4) \,\, = \,\,
{\rm det}(C(f)) \cdot {\rm disc}(f) \cdot {\rm Bdisc}(f). $$
\end{conjecture}    

The construction of the Blekherman discriminant extends to the case
when $f$ is a sextic or octic; see Lemma \ref{Blekhermangeneral}.
For quartics $f$, we can use
it to prove $\rrk(f) > 6$ also when the
 signature is $(4,2)$ or $(3,3)$. 
 We illustrate this for the quartic given by four distinct lines.

 \begin{example} \rm
 We claim that $f=xyz(x{+}y{+}z)$ has
  $\rrk(f) = 7$. For the upper bound, note
 $$  12 f \,\, = \,\, x^4\,+\,y^4\,+\,z^4 - \,(x+y)^4 - (x+z)^4 - (y+z)^4 +  (x+y+z)^4 . $$
The catalecticant $C(f)$ has signature $(3,3)$. The 
anti-polar quartic $\BQ(f)= a^2 b^2 + a^2 c^2 + b^2 c^2 - a^2 bc - ab^2c - abc^2  $ 
is nonnegative. By \cite[Section 5.1]{BBO}, we have $\rrk(f)=7$.
\hfill $ \diamondsuit$
\end{example}

If $f$ is a general ternary quartic of real rank $6$ then
${\rm SSP}(f)_\RR$ is an open semialgebraic set inside the threefold ${\rm VSP}(f)_\RR$.
Our next goal is to derive an algebraic description of this set.
We begin by reviewing some of the
 relevant algebraic geometry found in \cite{Kapustka, Muk92, Muk, RS, Sch}.

The Fano threefold ${\rm VSP}(f)$ in its anticanonical embedding is a subvariety of 
the Grassmannian ${\rm Gr}(4,7)$ in its Pl\"ucker embedding in $\PP^{34}$.  
It parametrizes  $4$-dimensional  subspaces of $f^\perp_3 \simeq \RR^7$
 that can serve to span $I_T$. In other words,
the Fano threefold ${\rm VSP}(f)$ represents quadruples of cubics
in $f^\perp_3 $ that arise from
a $3 \times 3 \times 4$-tensor $T$ as described above.
Explicitly,  ${\rm VSP}(f)$ is the intersection of ${\rm Gr}(4,7)$ with a  linear
subspace $\PP^{13}_A$ in $\PP^{34}$.
This is analogous to Proposition \ref{prop:Mukai35}, but more complicated.
The resolution of the apolar ideal $f^\perp$ has the form
$$0\longrightarrow S(-7)\longrightarrow S(-4)^7
\stackrel{A}{\longrightarrow} S(-3)^7\longrightarrow S\longrightarrow 0.$$
By the Buchsbaum-Eisenbud Structure Theorem, we can write $A=xA_1{+}yA_2{+}zA_3$
where $A_1,A_2,A_3$ are real skew-symmetric $7 {\times} 7$-matrices.
In other words, the matrices $A_1,A_2,A_3 $ lie in $ \bigwedge^2 f^\perp_3$.
The seven cubic generators of the ideal $f^\perp$ are the
$6 {\times} 6$-subpfaffians of~$A$.

The ambient space $\PP^{34}$ for the Grassmannian ${\rm Gr}(4,7)$
is the projectivation of the $35$-dimensional 
vector space $\bigwedge^4 f^\perp_3$. The matrices
$A_1,A_2,A_3$ determine the following subspace:
\begin{equation}
\label{eq:P13A}
\PP^{13}_A \,\,:=\,\bigl\{ U \in \bigwedge^4 f^\perp_3 \,:\, U\wedge A_1= U\wedge A_2= U \wedge A_3=0
\quad {\rm in} \,\,\,\bigwedge^6 f^\perp_3 \bigr\}.
\end{equation}
Each constraint $U  \wedge A_i = 0$ gives seven
linear equations, for a total of $21$ linear equations.

\begin{lemma} \label{prop:Mukai47}
The Fano threefold $\,{\rm VSP}(f)$ of degree $22$ is the
intersection of the Grassmannian ${\rm Gr}(4,7)$ with the
linear space $\PP^{13}_A$.
Its defining ideal is generated by $45$ quadrics, namely
the $140$
quadratic Pl\"ucker relations defining ${\rm Gr}(4,7)$
modulo the $21$ linear relations in~{\rm (\ref{eq:P13A})}.
\end{lemma}

\begin{proof}
This description of the Fano threefold $V_{22}$
was given by Ranestad and Schreyer in \cite{RS} and by
Dinew, Kapustka and Kapustka in \cite[Section 2.3]{Kapustka}.
We verified the numbers $22$ and $45$ by a direct computation.
\end{proof}

It is important to note that we can turn this construction around
and start with any three general skew-symmetric $7 \times 7$-matrices 
$A_1,A_2,A_3$. Then the $6 \times 6$-subpfaffians of $xA_1+yA_2+zA_3$
generate a Gorenstein ideal whose socle generator is a ternary quartic $f$.

This correspondence shows how to go from rank $6$ decompositions of $f$ 
to points $U$ in  ${\rm VSP}(f)   = {\rm Gr}(4,7) \cap \PP^{13}_A \,\subset\, \PP^{34}$. 
Given the configuration $\mathbb{X} = \{(a_i:b_i:c_i)\}$
in (\ref{eq:admitsmany}), the point
 $U$ is the space of cubics that vanish on $\mathbb{X}$. 
 Conversely, given any point $U$ in ${\rm VSP}(f)$,  we can choose
     a basis of $f^\perp_3$ such that
      our $7 \times 7$-matrices $A_i$ have the form
analogous to   (\ref{eq:specialA55}):
$$A_i \,\, = \,\, \begin{pmatrix}
\star & T_i \, \,\\
-T_i^t&0 \,\,\\
\end{pmatrix} \qquad \hbox{for}\, \,i=1,2,3.
$$
Here $T_i$ is a $4 \times 3$-matrix, and $0$ is the zero $3 \times 3$-matrix.
The four $3 \times 3$-minors of 
the $4 \times 3$-matrix $T = xT_1+yT_2+zT_3$
are among the seven $6 \times 6$-subpfaffians of $A  = xA_1+yA_2+zA_3$.
These four cubics define the six points 
in  the decomposition (\ref{eq:admitsmany}).

We are now ready to extend the real geometry in Corollary~\ref{cor:inaffine} from quadrics to quartics. 
Let $V = (v_{ij}) $ be a $4 {\times} 3$-matrix of unknowns.
These serve as affine coordinates on ${\rm Gr}(4,7)$.
Each point is the row span of the
 $4 {\times} 7$-matrix $U = \begin{pmatrix} \,{\rm Id}_4 \!& \! V\,\, \end{pmatrix}$.
 This is analogous to (\ref{eq:affinecoord}).
 
 Proceeding as in (\ref{eq:wenowtransform}), we consider the
   skew-symmetric $7 \times 7$-matrix
\begin{equation}
\label{eq:VAV}
\begin{pmatrix}   \star & T \,\, \\       -T^t & 0\,\,
\end{pmatrix} \quad = \quad 
\begin{pmatrix}\, {\rm Id}_4 & V \ \\
                          \,      V^t &  \! - {\rm Id}_3 
  \end{pmatrix} \cdot
A \cdot \begin{pmatrix} \,{\rm Id}_4 & V\ \\
                           \,     V^t &  \! - {\rm Id}_3 
  \end{pmatrix}.
 \end{equation}
  Its entries
are linear forms in $x,y,z$ whose coefficients are quadratic polynomials
in the $12$ affine coordinates $v_{ij}$.
Vanishing of the lower right $3 \times 3$-matrix 
defines ${\rm VSP}(f)$.
  The matrix $T$ is identified with a
  $4 \times 3 \times 3$-tensor whose entries
  are quadratic polynomials in the $v_{ij}$.
  
  \begin{theorem}
  Let $f$ be a general ternary quartic of real rank $6$.
  Using the affine coordinates $v_{ij}$ on ${\rm Gr}(4,7)$, the
    threefold ${\rm VSP}(f)_\RR$ is defined by nine 
     quadratic equations
    in $\RR^{12}$.
          If $f$ has signature $(6,0)$ then ${\rm SSP}(f)_\RR $ equals
$   {\rm VSP}(f)_\RR$. If $\overline{{\rm SSP}(f)_\RR}$ is a proper subset
of ${\rm VSP}(f)_\RR$ then 
its algebraic boundary has degree $84$. It is the
      hyperdeterminant of the $4 \times 3 \times 3$-tensor~$T$.
    \end{theorem}

\begin{proof}
The description of ${\rm VSP}(f)$ in affine coordinates
follows from Lemma \ref{prop:Mukai47}.
The equations in (\ref{eq:P13A})
mean that the linear map given by $A$ vanishes on the
kernel of $U$. This translates into the condition that the
lower right $3 \times 3$-matrix in (\ref{eq:VAV}) is zero.
Each of the $3$ coefficients of the $3$ upper diagonal matrix entries
must vanish, for a total of $9$ quadratic equations.

If $f$ has signature $(6,0)$ then we know from 
Proposition \ref{lem:2468} that
${\rm SSP}(f)_\RR =      {\rm VSP}(f)_\RR$.
In general, a point $(v_{ij})$ of ${\rm VSP}(f)_\RR$ lies in ${\rm SSP}(f)_\RR$ if and only if all six zeros
of the ideal $I_T$ are real points in $\PP^2$. The boundary of ${\rm SSP}(f)_\RR$ is given by those
$(v_{ij})$ for which  two of these zeros come together in $\PP^2$ and form a complex conjugate pair. 
The Zariski closure of that boundary is the hypersurface defined by the hyperdeterminant ${\rm Det}(T)$, 
by Proposition~\ref{prop:schlafli}.

The hyperdeterminant of format $4 \times 3 \times 3$ has degree $48$ in the tensor entries.
For our tensor $T$,
the entries are inhomogeneous polynomials of degree $2$, so the degree of ${\rm Det}(T)$
is bounded above by $96 = 2 \times 48$. A direct computation reveals that the
actual degree is $84$. The degree drop from $96$ to $84$ is analogous to the
 drop from $24$ to $20$ witnessed in (\ref{eq:schlafli}).
\end{proof}

At present we do not know whether the hyperdeterminantal boundary always exists:

\begin{conjecture}
If  the quartic $f$ has real rank $6$ and its signature is $(3,3)$, $(4,2)$ or $(5,1)$
then the semialgebraic set $\,\overline{{\rm SSP}(f)_\RR} $ is strictly contained in
the variety ${\rm VSP}(f)_\RR$.
\end{conjecture}

Our next tool for studying ${\rm VSP}(f)$ is another quartic curve,
derived from $f$, and endowed with a distinguished even theta
characteristic $\theta$.
 Recall that there is a unique (up to scaling) invariant of
 ternary cubics in degree $4$. This is the {\em
 Aronhold invariant}, which vanishes on cubics $g$ with $\crk(g) \leq 3$.
For the given quartic $f$,  the 
{\em Aronhold quartic} $\SSS(f)$ is defined by
\begin{equation}
\label{eq:AQ} \SSS(f)(p)\,\,:= \,\,
\hbox{the Aronhold invariant evaluated at}\,\,\partial_p(f).
\end{equation}
Following \cite{DK}, we call $f$ the {\em Scorza quartic} of $\SSS(f)$. 
Points on $\SSS(f)$ correspond to lines in the threefold ${\rm VSP}(f)$.
To see this, consider any decomposition $f=\sum_{i=1}^6 \ell_i^4$,
representing a point in ${\rm VSP}(f)$. This point lies on a 
$\PP^1 $ in $ {\rm VSP}(f)$ if and only if three of 
the lines $\ell_1,\ell_2,\ldots,\ell_6$ meet.
 Indeed, if $a \in \ell_1 \cap \ell_2 \cap \ell_3$ in $\PP^2$  then
  $\partial_p(f)$ is a sum of three cubes, i.e.~$p \in \SSS(f)$. 
 We may regard  $\ell_1,\ell_2,\ell_3$  as linear forms in two variables, so that
  $\ell_1^4+\ell_2^4+\ell_3^4$ is a binary quartic. 
  This binary quartic has a $\PP^1$ of rank $4$ decomposition, each
   giving a decomposition of $f$, with $\ell_4,\ell_5,\ell_6$ fixed. The resulting
    line in ${\rm VSP}(f) \subset  \PP^{13}_A$ is the
   set of $4$-planes $U $ containing the $\PP^3$ of cubics $ Q \cdot a$, 
   where $Q$ is a quadric vanishing on $\ell_4,\ell_5,\ell_6$.
   This gives   all lines on ${\rm VSP}(f)$.
   
One approach we pursued is the relationship   
of the real rank of $f$ with its  topology in $\PP_\RR^2$.
A classical result of Klein and Zeuthen, reviewed in
\cite[Theorem 1.7]{PSV}, states that there are six types of smooth plane
quartics in $\PP^2_\RR$, and
these types form connected subsets of $\PP^{14}_\RR$:
\begin{equation}
\label{eq:sixtypes}
 \hbox{\rm 4 ovals,  \ 3 ovals, \ 2 non-nested ovals, \ hyperbolic, \ 1 oval, \ empty curve.} 
 \end{equation}
We consider the pairs of types given by a 
 general  quartic $f$ and its Aronhold quartic $\SSS(f)$.
  
 \begin{proposition}
 Among the $36$ pairs of topological types {\rm (\ref{eq:sixtypes})} of
 smooth quartic curves in the real projective plane $\PP^2_\RR$, at least $30$
 pairs are realized by a quartic $f$ 
 and its Aronhold quartic $\SSS(f)$.
 Every pair not involving the hyperbolic type is 
 realizable as $\bigl(f,\SSS(f) \bigr)$.
  \end{proposition}
 
 \begin{proof}
 This was established by exhaustive search.
 We generated random quartic curves using various
 sampling schemes, and this led to
$30$ types. The six missing types are listed in Conjecture
\ref{conj:6missing}. For a concrete example,
here is an instance where $f$ and $\SSS(f)$ are empty:
\begin{small}
$$ 
 f = 
(3 x^2 + 5 z x- 5 y x-5 z^2-3 y z)^2+(7 x^2+7 z x-7 y x+z^2-y z-5 y^2)^2+(5 x^2+ 
7 z x+7 y x-8 z^2-2 y z+2 y^2)^2 
$$
\end{small}
At the other end of the spectrum, let us consider
$$
\begin{matrix}  f & =  & 
1439 x^4+1443 y^4-2250 (x^2+y^2) z^2+3500 x^2 y^2+817 z^4-x^3 z-x^3 y
\\ & & 
-5 x^2 z^2 -7 x^2 y z-4 x z^3-6 x y z^2+x y^2 z+6 x y^3-3 y z^3+5 y^2 z^2+7 y^3 z
\end{matrix}
$$
For this quartic, both $f$ and $\SSS(f)$ have 
$28$ real bitangents, so they consist of four ovals.
 \end{proof}
  
 \begin{conjecture} \label{conj:6missing}
 If a smooth quartic $f$  on $\PP^2_\RR$
 is hyperbolic then its Aronhold quartic $\SSS(f)$ is either empty or has two ovals.
 If $f$ consists of three or four ovals then $\SSS(f)$ is not hyperbolic.
 \end{conjecture}

We now describe the construction in \cite{Muk} 
of eight distinguished {\em Mukai decompositions}
\begin{equation}
\label{eq:specialdecomp}
 f \,\, = \,\, \ell_1^4 + \ell_2^4 + \ell_3^4 + 
\ell_{12}^4 + \ell_{13}^4 + \ell_{23}^4. 
\end{equation}

The configuration $\ell_{12},\ell_{23},\ell_{13}$ is a
{\em biscribed triangle} of the Aronhold quartic $\SSS(f)$.
Being a biscribed triangle means that $\ell_{ij}$ 
is tangent to $\SSS(f)$ at a point $q_{ij}$, 
the lines $\ell_{ij}$ and $\ell_{ik}$ meet at a point $q_i$ on the curve $\SSS(f)$,
and the line $\ell_i$ is spanned by $q_{ij}$ and $q_{ik}$.


Let $D = q_1 + q_2 + q_3 + 
q_{12} + q_{13} + q_{23}$ be a divisor on the Aronhold quartic $\SSS(f)$.
The biscribed triangle $\ell_{12} \ell_{13} \ell_{23}$
is a {\em contact cubic} \cite[\S 2]{PSV},
and $2D$ is its intersection divisor with $\SSS(f)$. 
The associated theta characteristic is given by 
$\theta \sim   q_{12} + q_{13}  + q_{1} - q_{23}$.
Each of the points $q_{12},q_{13},q_{23} \in \SSS(f)$
represents  a line on the Fano threefold ${\rm VSP}(f) \subset
\PP^{13}_A$.  The pairs $(q_{12},q_{13})$,
$(q_{12},q_{23})$, $(q_{13},q_{23})$ lie in
the {\em Scorza correspondence}, as defined in \cite{DK, Sch}.
Indeed,   the  corresponding
second derivatives of $f$ are $\ell_{12}^2$,
$\ell_{13}^2$ and $\ell_{23}^2$, so
the  lines $q_{12},q_{13},q_{23}$
on ${\rm VSP}(f)$ intersect pairwise. In fact,
there is a point of ${\rm VSP}(f)$ on all
three lines, namely~(\ref{eq:specialdecomp}).

\begin{example} \rm
We illustrate the concepts above, starting with the skew-symmetric matrix
$$
A = A_1 x + A_2 y + A_3 z \, = 
\begin{small} \begin{pmatrix}
    0  &    -x{+}y{+}3z & z  &    y+z  &   x &   -x &   0   \\
   \! x{-}y{-}3z & 0  &     x{-}y{+}3z & -x{-}3y{+}z & -x+z &  2x{-}2y &  y{-}z \\
       -z  &   -x{+}y{-}3z &  0  &    x{+}y{+}z  &  0 &   y  &   -y  \\
    -y-z  &  x{+}3y{-}z & -x{-}y{-}z  & 0  &     z &   -3z &  z   \\
     -x  &   x-z &    0 &     -z  &    0 &   0 &    0   \\
       x  &    -2x+2y &  -y  &   3z  &     0 &   0 &    0   \\
       0  &    -y+z  &   y  &    -z  &    0 &   0  &   0   
\end{pmatrix}\!.
\end{small}
$$
Its $6 \times 6$ pfaffians generate the apolar ideal $f^\perp$. Orthogonal to this is the
rank $6$ quartic
$$ f \,\,=\,\, x^4+y^4+z^4+(x+y)^4+(y+z)^4+(z+x)^4 .$$
The  upper right $4 \times 3$-block of $A$ has rank $2$ precisely on these six points 
$\ell_1,\ell_2,\ell_3,\ell_{12},\ell_{13},\ell_{23}$.
Here, $ q_1 = (-1{:}1{:}1), q_2 = (1{:}{-}1{:}1), q_3 = (1{:}1{:}{-}1),
q_{12} = (0{:}0{:}1), q_{13} = (0{:}1{:}0), q_{23} = (1{:}0{:}0)$.      
The theta characteristic $\theta$ on
the Aronhold quartic $\SSS(f)$ is defined by the contact cubic
 $ (x+y)(x+z)(y+z)$.
 This is the lower right $3 \times 3$-minor in the 
  determinantal representation
$$ \SSS(f) \,=\,
{\rm det} \! \begin{pmatrix}
\,5 x + 5 y + 5 z &  x - y   &    x - z &   y - z  \\
             x - y        &  x + y + z &      x   &    -y    \\
             x - z        &    x    & x + y + z &    z    \\
              y - z        &    -y    &      z   & x + y + z\,
              \end{pmatrix}.
              $$
This matrix is constructed from the contact cubic by the method
in  \cite[Proposition 2.2]{PSV}.
\hfill $ \diamondsuit$
\end{example}

We write
${\rm VSP}(f)^{\rm Muk}$ for the subvariety of ${\rm VSP}(f)$
given by Mukai decompositions~(\ref{eq:specialdecomp}).
Mukai \cite{Muk} showed that ${\rm VSP}(f)^{\rm Muk}$ is a finite
set with eight elements. We are interested in
$$
{\rm VSP}(f)_\RR^{\rm Muk} \,:=\,  {\rm VSP}(f)^{\rm Muk} \,\cap\, {\rm VSP}(f)_\RR
\quad {\rm and} \quad
{\rm SSP}(f)_\RR^{\rm Muk} \,:= \, {\rm VSP}(f)^{\rm Muk}\, \cap\, {\rm SSP}(f)_\RR.
$$
One idea we had for certifying $\rrk(f) = 6$ is to compute the eight points in
${\rm VSP}(f)^{\rm Muk}$. If  (\ref{eq:specialdecomp}) is fully
real for one of them then we are done. Unfortunately, this algorithm may fail.
The semialgebraic set of quartics with real Mukai decompositions
is strictly contained  in  $\mathcal{R}_4$:
 
\begin{proposition} \label{prop:mukai}
There exist quartics $f$ of real rank $6$ such that
$\,{\rm SSP}(f)^{\rm Muk}_\RR$ is  empty.
 \end{proposition}

\begin{proof}
Consider the eight  Mukai decompositions (\ref{eq:specialdecomp})  
of  the following quartic of real rank $6$:
$$ 
\begin{matrix} f &  = & 
(21x+y+9z)^4+(14x-13y+14z)^4+(13x+5y-7z)^4 
\\ & & +\,(2x-5y-13z)^4
\,-\,(12x+15y-9z)^4 \,-\,(12x+21z)^4.
\end{matrix} $$
The lines $\ell_i$ and $\ell_j$ in (\ref{eq:specialdecomp})
 intersect in the point $q_{ij} \in S(f)$.
If both lines are real then so is $q_{ij}$.
But, a computation shows that $S(f) $ has no real points.
This implies $\,{\rm SSP}(f)^{\rm Muk}_\RR = \emptyset$.
\end{proof}

\begin{figure}[h]
    \centering
          \includegraphics[scale=0.25]{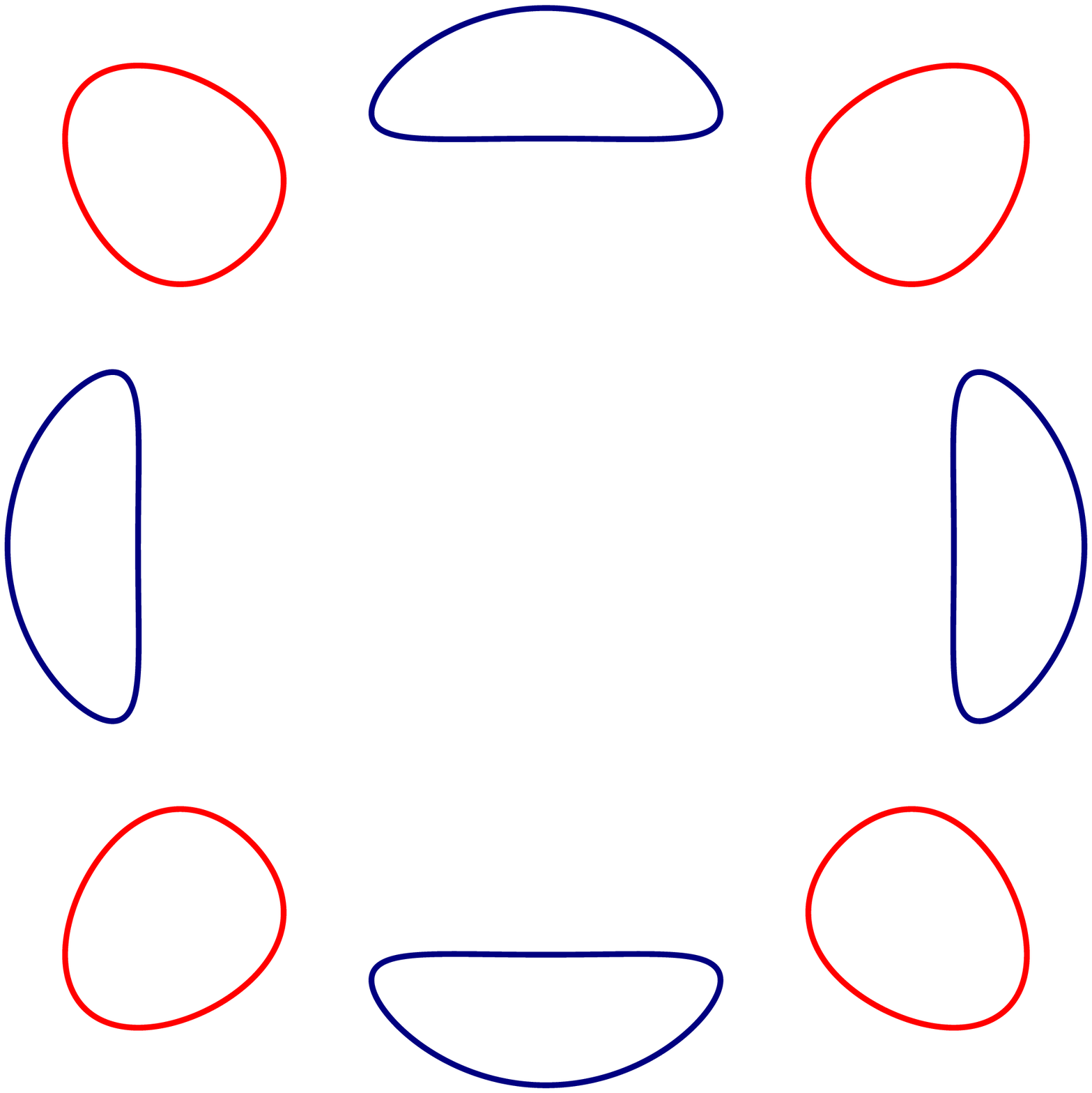} \qquad \qquad \quad
        \includegraphics[scale=0.25]{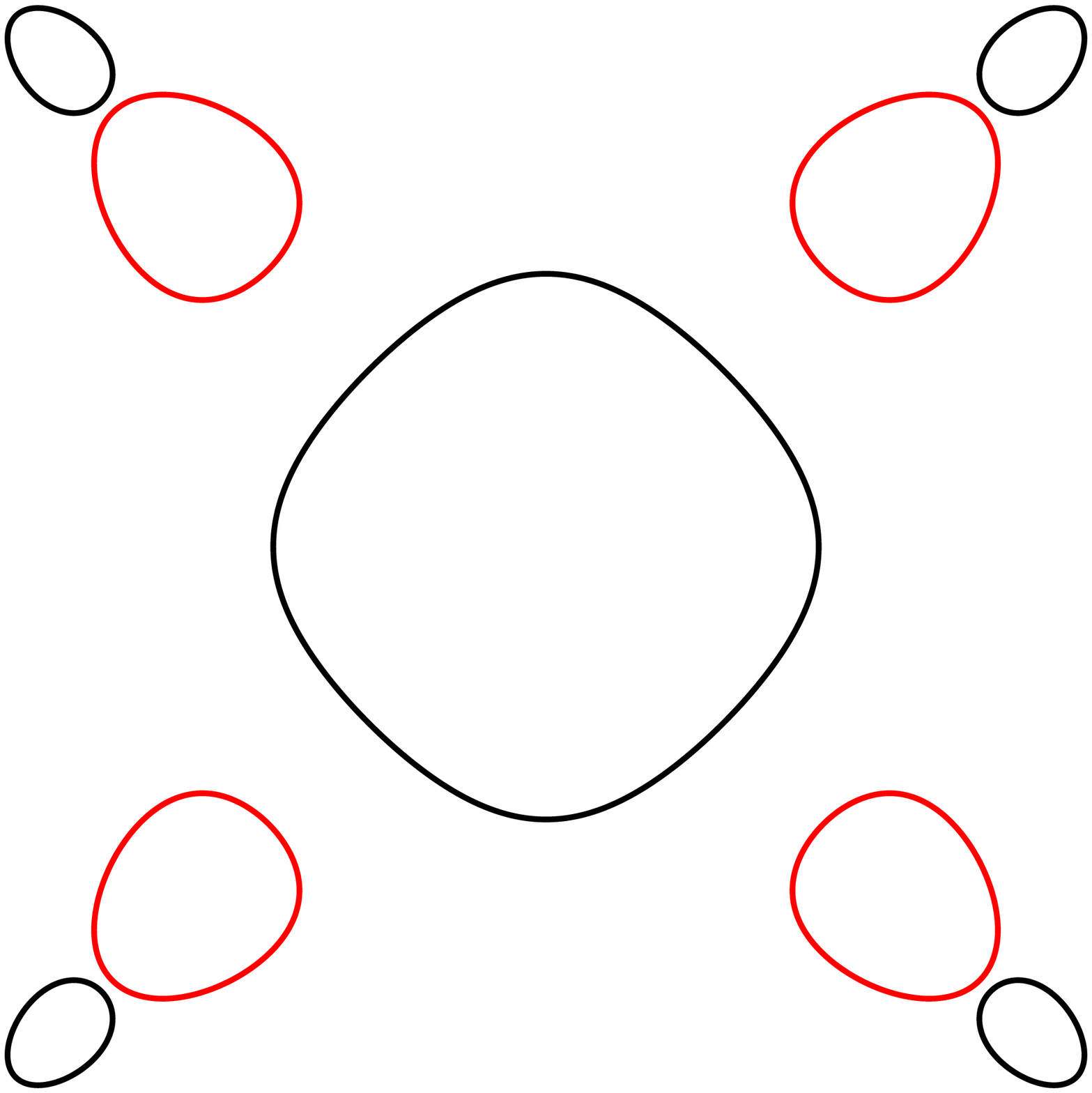}
        \caption{
 The left picture shows a Blum-Guinand quartic in blue and its Aronhold quartic in red.
 The Aronhold quartic does not meet the  sextic covariant, shown in black on the right.}
        \label{fig: BlumAronhold}
\end{figure}

\begin{example} \rm
We close this section by discussing the {\it Blum-Guinand quartics} in \cite{BlumGuinand}. Set
$$
f_{a,b,m}\,\,:=\,\,(x^2+y^2-a^2z^2-b^2z^2)(m^4x^2+y^2-m^2(a^2+b^2)z^2)+a^2b^2(m^2-1)^2z^4.
$$
The parameters $a,b,m$
satisfy $0<b<a$ and $\sqrt{b/a}< m < \sqrt{a/b}$. 
Blum-Guinand quartics have $28$ real bitangents, so they consist of four ovals.
 Figure \ref{fig: BlumAronhold} shows this for $a=20$, $b=8$ and $m=\sqrt{\frac{20}{8}}-0.1$.
  The diagram on the left has $f_{a,b,m}$ in blue and $S(f_{a,b,m})$ in red.
   The sextic covariant, defined by (\ref{eq:AQ}) but with $S$ replaced by
    the sextic invariant $T$,  is shown on the right in black.
   The cubic $ \partial_p(f_{a,b,m})$ has real rank $4$ for all $p \in S(f_{a,b,m})$;
   cf.~Remark \ref{rem:Hessereal}. 
 For the chosen parameters, for each direction there exists a line that meets the Blum quartic at 4 real points. 
   We can conclude that $\rrk(f_{a,b,m}) \geq 7$. In general, the signature is (4,2) if $\sqrt{2}-1<m<\sqrt{2}+1$ and (3,3) otherwise. 
   The picture seen in Figure \ref{fig: BlumAronhold} can change.
    For instance,      the Aronhold quartic $S(f_{a,b,m})$ has no real points
   when $a=70, b=8, m=6/5$. 
   \hfill $\diamondsuit$
   \end{example}
   
\section{Quintics and Septics}

A general ternary quintic 
$f\in \RR[x,y,z]_5$ has
complex rank $  R(5)=7$. The decomposition
 \begin{equation}
 \label{eq:fseven}
  f \,\, = \,\,
  \ell_1^5 +   \ell_2^5 +   \ell_3^5 +   \ell_4^5 +   \ell_5^5 +   \ell_6^5 +   \ell_7^5 
  \end{equation}
is unique by a classical result of Hilbert, Richmond and Palatini.
Oeding and Ottaviani \cite{OO} explained how to compute the 
seven linear forms $\ell_i$ by realizing them as  eigenvectors
of a certain $3 \times 3 \times 3$-tensor.
Inspired by \cite[\S 1.5]{RS},
we propose the following alternative~algorithm:

\begin{algorithm}\label{Decomposition of quintic}
\underbar{\rm Input}: A general ternary quintic $f$.
\underbar{\rm Output}: The decomposition {\rm (\ref{eq:fseven})}.

\begin{enumerate}
\item Compute the apolar ideal $f^{\perp}$.
It is generated by one quartic and four cubics $g_1,g_2,g_3,g_4$.
\item Compute the syzygies of $f^{\perp}$. 
Find the unique linear syzygy $(l_1,l_2,l_3,l_4)$ on the cubics.
\item  Compute a vector $(c_1,c_2,c_3,c_4) \in \RR^4 \backslash \{0\}$
that satisfies $c_1 l_1 + c_2 l_2 + c_3 l_3 + c_4 l_4 = 0$.
\item Let $J$ be the ideal  generated by the cubics  $\,c_2 g_1 - c_1 g_2$, 
$\,c_3 g_2 - c_2 g_3\,$ and $\,c_4 g_3 - c_3 g_4$. Compute the variety $V(J)$ in $\PP^2$.
It consists precisely of the points dual to $\ell_1,\ell_2,\ldots,\ell_7$.
\end{enumerate}
\end{algorithm}

To prove the correctness of this algorithm, we  recall
what is known about the ideal $J$ of seven points in $\PP^2$.
The ideal $J$ is Cohen-Macaulay of codimension $2$, so it is
generated by the maximal minors of its Hilbert-Burch matrix $T$.
According to  \cite[Theorem 5.1]{ASS}, this matrix has the following form
 if and only if no six of the seven points lie on a conic:
\begin{equation}
\label{eq:onaconic}
 T\,=\,
 \begin{pmatrix}
 q_1& l_1\\
 q_2& l_2\\
 q_3& l_3
 \end{pmatrix}.
\end{equation}
Here $l_1,l_2,l_3$ are independent linear forms and $q_1,q_2,q_3$ are 
quadratic forms in $x,y,z$.

\begin{proposition}
 Algorithm \ref{Decomposition of quintic} computes the unique decomposition of a general quintic $f$. 
 In the resulting  representation (\ref{eq:fseven}),
   no six of the seven lines $\ell_i$ are tangent to a conic. 
   \end{proposition}

\begin{proof}
Let $f$ be a general quintic. The apolar ideal $f^{\perp}$
in $S$ is generated by four cubics 
$g_1,g_2,g_3, g_4$ and one quartic $h$. This  ideal is Gorenstein of codimension $3$.
 The Buchsbaum-Eisenbud Structure Theorem implies that the
minimal free resolution of $f^{\perp}$ has the form 
$$ 0 \longrightarrow S(-8)\longrightarrow S(-4)\oplus S(-5)^4 
\stackrel{A}{\longrightarrow} 
S(-4)\oplus S(-3)^4 \stackrel{B}{\longrightarrow} S \longrightarrow 0.
$$
The matrix $A$ is skew-symmetric of size $5\times 5$, i.e. 
\begin{equation}
\label{eq:Nmatrix}
 A \,\,=\,\,
\begin{pmatrix}
   0 & q_{12} & q_{13}& q_{14} & l_1\\
   -q_{12} & 0 & q_{23} & q_{24} & l_2\\
   -q_{13} & -q_{23} & 0 & q_{34} & l_3\\
   -q_{14} & -q_{24} & -q_{34} & 0 & l_4\\
   -l_1 & -l_2 & -l_3 & -l_4 & 0
\end{pmatrix}.
\end{equation}
\noindent Here the $l_i$'s are linear forms and 
the $q_{ij}$'s are quadrics. As described above and
in Section~2, we should find a $5\times5$ matrix $U$ such that the lower right $2\times2$ submatrix of $U\cdot A\cdot U^t$ is the zero matrix. Since $f$ is general, we may assume that 
the $l_i$ span $\RR[x,y,z]_1$. After relabeling if necessary, we can write
 $c_1l_1+c_2l_2+c_3l_3+l_4=0$ for some scalars $c_1,c_2,c_3$.
 Setting $c_4 = 1$, these are  the scalars
   in Step 3 of 
Algorithm  \ref{Decomposition of quintic}.~Let
$$ U\, = \,
\begin{pmatrix}
1&0&0&0&0\\
0&1&0&0&0\\
0&0&1&0&0\\
c_1&c_2&c_3&1&0\\
0&0&0&0&1
\end{pmatrix}. 
$$
\noindent We perform row and column operations by the following right and left multiplication:
$$
A'\,\,= U\cdot A \cdot U^t=\,\,\begin{pmatrix}
   0 & q_{12} & q_{13}& q_{14}' & l_1\\
   -q_{12} & 0 & q_{23} & q_{24}' & l_2\\
   -q_{13} & -q_{23} & 0 & q_{34}' & l_3\\
   -q_{14}' & -q_{24}' & -q_{34}' & 0 & 0\\
   -l_1 & -l_2& -l_3& 0 & 0
\end{pmatrix}.
$$
The inverse column operation on the row 
vector $B$ of minimal generators of $f^{\perp}$ gives
$$
B'\,=\,
\begin{pmatrix}
g_1 \! & \! g_2 \! & \! g_3 \! & \! g_4\!& \! h
\end{pmatrix} \!
\begin{pmatrix}
1&0&0&0&0\\
0&1&0&0&0\\
0&0&1&0&0\\
\! -c_1& \! -c_2&\! -c_3&1&0\\
0&0&0&0&1
\end{pmatrix} \,=\,
\begin{pmatrix}
 g_1{-}c_1g_4& g_2{-}c_2 g_4& g_3{-}c_3 g_4 & g_4&h
\end{pmatrix}\!.
$$

 Let $\,J = \langle g_i -c_i g_4 \,:\,i=1,2,3\rangle\,$
  denote the ideal generated by the first three cubics.
 This is the ideal in Step 4 of the algorithm.
 We claim that $V(J)$ consists of seven points in $\PP^2$.

By construction, we have $B' \cdot A' = 0$, and the columns of $A'$
span the syzygies on $B'$. The entries of $B'$ are the 
$4 \times 4$ subpfaffians of $A'$. The first three entries
are the subpfaffians that involve the last two rows and columns.
These three $4 \times 4$ pfaffians are the $2 \times 2$-minors~of
$$
T\,\,=\,\begin{pmatrix}
   q_{14}' & l_1 \\
   q_{24}' & l_2\\
   q_{34}' & l_3
\end{pmatrix}.
$$
\noindent This is a Hilbert-Burch matrix for the ideal $J$.
Hence $J$ is
  an ideal of seven points in $\PP^2$. Moreover, since $l_1, l_2$ and $l_3$ are linearly independent, no six points of them lie on a conic. 
  Dually, this means that no six of the seven lines $\ell_i$ used in
(\ref{eq:fseven}) are tangent to a conic.
\end{proof}
 
It is easy to decide whether the real rank of a given ternary quintic $f$ is $7$ or not.
Namely, one computes the unique complex decomposition
(\ref{eq:fseven}) and checks whether it is real.
The real rank boundary corresponds to transition points
where two of the linear forms in (\ref{eq:fseven}) come together
and form a complex conjugate pair. The following is
our main result on quintics.

\begin{theorem}
\label{thmfseventan}
The algebraic boundary $\partial_{\rm alg}( \mathcal{R}_5)$
of the set $\mathcal{R}_5 = \{f:\rrk(f) = 7\}$ is an irreducible
hypersurface of degree $168$ in the $\PP^{20}$ of quintics.
It has the parametric representation
\begin{equation}
\label{eq:fseventan}
g \,\, = \,\,
\ell_1^5 + \ell_2^5 + \ell_3^5 + \ell_4^5 + \ell_5^5 \,+ \, \ell_6^4 \ell_7 ,
\qquad
\hbox{where $\,\ell_1,\dots,\ell_7 \in \RR[x,y,z]_1$.}
\end{equation}
\end{theorem}

\begin{proof}
The parametrization (\ref{eq:fseventan}) defines a unirational variety $Y$ in $\PP^{20}$.
The Jacobian of this parametrization is found to have corank $1$. This means
that $Y$ has codimension $1$ in $\PP^{20}$.
 Hence $Y$ is an irreducible hypersurface, defined by a unique (up to sign)
 irreducible homogeneous polynomial $\Phi$ in $21$ unknowns,
 namely the coefficients of a ternary quintic.

Let $g$ be a real quintic (\ref{eq:fseventan})
that  is a general point in $Y$. For $\epsilon \rightarrow 0$, the real quintics
$(\ell_6+\epsilon\ell_7)^5-\ell_6^5$ and 
$(i\ell_6+\epsilon \ell_7)^5+(-i\ell_6+\epsilon \ell_7)^5$
converge to the special quintic $\ell_6^4 \ell_7$  in $\PP^{20}_\RR$.
Hence any small neighborhood of $g$ in $\PP^{20}_\RR$ contains 
quintics of real rank $7$ and quintics of real rank $\geq 8$.
This implies that $Y$ lies in the 
algebraic boundary $\partial_{\rm alg}( \mathcal{R}_5)$.
Since $Y$ is irreducible and has codimension $1$, it follows
that $\partial_{\rm alg}( \mathcal{R}_5)$ exists and has
$Y$ as an irreducible component.

We carried out an explicit computation
 to determine that the (possibly reducible)
hypersurface $\partial_{\rm alg}( \mathcal{R}_5)$ has degree $168$.
This was done as follows.
Fix the field $K = \mathbb{Q}(t)$, where $t$ is a new variable.
We picked random quintics $f_1$ and $f_2$  in $\mathbb{Q}[x,y,z]_5$,
and we ran Algorithm~\ref{Decomposition of quintic} for  $f = f_1 + t f_2  \in K[x,y,z]_5$.
Step 4 returned a homogeneous ideal $J$ in $K[x,y,z]$ that defines
$7$ points in $\PP^2$ over the algebraic closure of $K$.
By eliminating each of the three variables, we obtain
binary forms of degree $7$ in $K[x,y]$, $K[x,z]$ and $K[y,z]$.
Their  coefficients are polynomials of degree $35$ in $t$.
The discriminant of each binary form is a polynomial in $\mathbb{Q}[t]$ of degree $420 = 12 \times 35$.
The greatest common divisor of these three discriminants is a
polynomial $\Psi(t)$ of degree $168$. We checked that $\Psi(t)$ is
  irreducible in $\mathbb{Q}[t]$.

By definition, $\Phi$ is an irreducible
homogeneous polynomial with integer coefficients in the $21$ coefficients
of a general quintic $f$. Its specialization $\Phi(f_1 + t f_2)$ is
a non-constant polynomial in $\mathbb{Q}[t]$, of degree ${\rm deg}(X)$ in $t$.
That polynomial divides $\Psi(t)$. Since the latter is irreducible, we conclude that
 $\Phi(f_1 + t f_2) = \gamma \cdot \Psi(t)$, where $\gamma$ is
 a nonzero rational number. Hence $\Phi$ has degree $168$.
We conclude that ${\rm deg}(Y) = 168$, and therefore
$ Y = \partial_{\rm alg}( \mathcal{R}_5)$.
\end{proof}

Theorem \ref{thmfseventan} was stated for a very special 
situation, namely ternary quintics. We shall now describe
a geometric generalization.
Let $X$ be any irreducible projective variety in the
complex projective space $\PP^N$ that is defined over $\RR$
and whose real points are Zariski dense.
We set $d = {\rm dim}(X)$.
The {\em generic rank} is the smallest integer $r$
such that the $r$th secant variety $\sigma_r(X)$ equals $\PP^N$.
Given $f \in \PP^N$, we define
 ${\rm VSP}_X(f)$
to be the closure in the Hilbert scheme ${\rm Hilb}_r(X)$
of the set of configurations of $r$ distinct points in $X$ whose span contains $f$.
Now, ${\rm VSP}$ stands for {\bf v}ariety of {\bf s}ums of {\bf p}oints.
This object agrees with that studied by Gallet, Ranestad and Villamizar in \cite{GRV}.
It differs from more inclusive definitions seen in other articles.
In particular, if $N = \binom{d+2}{2}-1$ and $X = \nu_d(\PP^2)$
is the $d$th Veronese surface then
${\rm VSP}_X(f) = {\rm VSP}(f)$. In this case, we recover the familiar
{\bf v}ariety of {\bf s}ums of {\bf p}owers.

The objects of real algebraic geometry studied in this paper
generalize in a straightforward manner.
We write ${\rm VSP}_X(f)_\RR$ for the variety of real points
in ${\rm VSP}_X(f)$, and we define ${\rm SSP}_X(f)_\RR$ to be
the semialgebraic subset of all $f$ that lie in an $(r-1)$-plane
spanned by $r$ real points in $X$. 
Following  Blekherman and Sinn \cite{BS}, we are interested in
 generic points in $\PP^N_\RR$ whose real rank equals the generic complex rank.
 These comprise the semialgebraic set
$\mathcal{R}_X = \{f \in \PP^N_\RR : {\rm SSP}_X(f)_\RR \not= \emptyset \}$.
The topological boundary $\partial \mathcal{R}_X$ is 
the closure of $\mathcal{R}_X$ minus the interior of that  closure.
If $X$ has more than one typical real rank then
$\partial \mathcal{R}_X$ is non-empty and its Zariski closure
$\partial_{\rm alg}( \mathcal{R}_X)$ is a hypersurface in $\PP^N$.
This hypersurface is the {\em real rank boundary} we are interested in.

\begin{example}
\label{ex:fifthvero}
 \rm
Let $N = 20$ and $X = \nu_5(\PP^2)$ the fifth Veronese surface in $\PP^{20}$.
Then $r= 7$ and $\partial_{\rm alg} (\mathcal{R}_X)$
equals the irreducible hypersurface of degree $168$ 
described in Theorem~\ref{thmfseventan}.
\hfill $ \diamondsuit$
\end{example}

This example  generalizes as follows. Let $X \subset \PP^N$ as above, 
and let $\tau(X)$ denote its {\em tangential variety}. By definition,
$\tau(X)$ is the closure of the union of all lines that are tangent to $X$.
We also consider the $(r-2)$nd secant variety $\sigma_{r-2}(X)$. 
The expected dimensions  are
 $$ {\rm dim}\bigl(\tau(X)\bigr) = 2 d \quad  {\rm and} \quad
 {\rm dim}\bigl(\sigma_{r-2}(X) \bigr) =  (r-2) d + r-3  . $$
We are interested in the join of the two varieties, denoted
$\sigma_{r-2}(X) \star \tau(X)$.
This is an irreducible projective variety of expected dimension $rd + r -2 $ in $\PP^N$.
It comes with a distinguished parametrization,
 generalizing that in (\ref{eq:fseventan}) for the
Veronese surface of Example~\ref{ex:fifthvero}.
The following generalization of Theorem~\ref{thmfseventan}
explains the geometry of the real rank boundary:
    
\begin{conjecture} \label{conj:star}
Suppose $rd+r = N$ and $\,{\rm VSP}_X(f)$ is finite for general $f$. Then
$\sigma_{r-2}(X) \star \tau(X)$ is an irreducible component of
$\partial_{\rm alg} (\mathcal{R}_X)$.
Equality holds when ${\rm VSP}_X(f)$ is a point.
\end{conjecture}

One difficulty in proving this conjecture is that we do not know 
how to control interactions among the distinct decompositions 
$f = f_1 + \cdots + f_r$ of a general point $f \in \PP^N$ into $r$ points
$f_1,\ldots,f_r $ on the variety $ X$. Moreover, we do not even
know that $ \partial_{\rm alg}( \mathcal{R}_X)$ is non-empty.

To illustrate Conjecture \ref{conj:star}, we prove it in 
the case when $X$ is the $7$th Veronese surface.
The parameters are $d=2, N = 35$, and $r=12$. We return to the previous
notation, so  $f $ is a general ternary form in $\RR[x,y,z]_7$.
Here we can show that $13$ is indeed a typical real rank.

\begin{proposition}
The real rank boundary 
$\partial_{\rm alg}( \mathcal{R}_7)$
is a non-empty hypersurface in $\PP^{35}$ with one of the components equal to the join of the tenth secant variety and the tangential variety.
\end{proposition}

\begin{proof}
For a general septic $f$, the
minimal free resolution of the apolar ideal $f^\perp$ equals
\begin{equation}
\label{eq:10640}
0\longrightarrow S(-10)\longrightarrow S(-6)^5
\stackrel{A}{\longrightarrow} S(-4)^5\longrightarrow S\longrightarrow 0.
\end{equation}
This is as in (\ref{eq:5320}), but now 
the entries of the skew-symmetric $5 \times 5$-matrix $A$
are quadratic. To find the five rank $12$ decompositions of $f$,
we proceed  as in Example \ref{ex:x^2+y^2-z^2}: we solve the matrix
equation (\ref{eq:wenowtransform}). The matrix entry in position $(4,5)$
is a homogeneous quadric in $x,y,z$ whose
six coefficients are inhomogeneous quadrics in the unknowns $a,b,c,d,e,g$.
These coefficients must vanish. This system of six equations
in six unknowns defines ${\rm VSP}(f)$ in $\CC^6$. It
has precisely five solutions. For each of these solutions
we consider the upper right $3 \times 2$-matrix~$T$. Its entries
are quadrics in $x,y,z$. The $2 \times 2$-minors of $T$ define 
the desired $12$ points in~$\PP^2$. 

We apply the above algorithm to the point in
$\,\sigma_{10}(X) \star \tau(X)\,$ given by the septic
$$ \begin{matrix}
f &:= & (11x+13y-12z)(-18x+13y-16z)^6+(2x-12y+z)^7
+(3x-13y-7z)^7\\ & & +(-6x+5y+15z)^7 
 +(16x+5y+14z)^7+(18x-19y-9z)^7
+(-4x+10y-18z)^7 \\ & & +(11x+9y+10z)^7 
 +(-19x+15y-z)^7+(-4x-20y-16z)^7+(2x+20y-18z)^7.
\end{matrix}
$$
This septic $f$ has complex rank $12$, but its real rank is larger. The output
of our decomposition algorithm shows
 that four of the five decompositions are not fully real. This remains
 true in a small neighborhood of $f$. Near the point $(11x+13y-12z)(-18x+13y-16z)^6 $ in the tangential variety $ \tau(X)$,
 some septics have real rank $2$ and some others have real rank $3$.
  Hence, the above decomposition of $f$ can change from purely real to a decomposition that contains complex linear forms. The same holds for all nearby points in the join variety.
  We conclude that a general point of the join in a small neighborhood of $f$ 
   belongs to $\partial( \mathcal{R}_7)$.
\end{proof}

Our algorithm for septics $f$ computes the five elements
in ${\rm VSP}(f)$ along with the $12$ linear forms in each of the  five decompositions
$f = \sum_{i=1}^{12} \ell_i^7$. It outputs $60$ points in $\PP^2$.
These come in $5$ unlabeled groups of $12$ unlabeled points in $\PP^2$.
Here are two concrete instances.

\begin{example} \rm
First, consider the septic $f=\sum_{i=1}^{12}\ell_i^7$ of real rank $12$ that is defined by $$
\begin{matrix}
\ell_1=-7x+14y+3z, & \ell_2= -13x-12y+20z,& \ell_3=-7x-5y-18z, \\
\ell_4=12x-16y+17z, & \ell_5=-8x+16y+7z,&\ell_6=15x-8y+2z ,\\
\ell_7=13x-7y-11z ,& \ell_8=-x-3y-3z, & \ell_9=18x-4y-7z ,\\
\ell_{10}=19x-9y+7z ,& \ell_{11}=15y-17z ,& \ell_{12}=-19x-2y+11z.
 \end{matrix}
 $$
 Here, all $60$ points in $\PP^2$ are real. This means that
 the variety of sums of powers is fully real,
 and the twelve $\ell_i$ in each of the five decompositions are real:
 ${\rm VSP}(f) = {\rm VSP}(f)_\RR={\rm SSP}(f)_\RR$. 
 
Second, consider the real septic 
$f=\sum_{i=1}^6(\ell_i^7+\bar{\ell}_i^7)$ defined by the complex linear forms
$$
\begin{matrix}
\ell_1=8x+11y-15z+i(13x+15y+17z) , & \ell_2=-16x+6y+11z+i(-4x+19y+16z),\\
\ell_3=5x+13y-3z+i(-2x+4y+5z) ,&\ell_4=8x+8y-7z+i(-13x-12y-8z),\\
\ell_5=-5x-20y-15z+i(-8x+18y+7z),& \ell_6=-14x-18y-7z+i(-9x-2y+19z).
\end{matrix}
$$
This satisfies  ${\rm VSP}(f)_\RR \neq {\rm SSP}(f)_\RR=\emptyset$. 
None of the $60$ points in $\PP^2$ are real.
Our two examples are extremal.
One can easily find other septics $f$ 
with ${\rm VSP}(f)_\RR \neq {\rm SSP}(f)_\RR$.
\hfill $\diamondsuit$
\end{example}

\section{Sextics}

We now consider ternary forms of degree six.
The generic complex rank for sextics is $R(6) = 10$.
Our first result states that both $10$ and $11$ are 
typical real ranks, in the sense of \cite{BBO, Ble, CO}.

\begin{theorem} \label{thm:severi}
The algebraic boundary 
$\partial_{\rm alg}( \mathcal{R}_6)$ is a hypersurface
in the $\PP^{27}$ of ternary sextics. One of its irreducible
components is the dual to the
Severi variety of rational sextics.
\end{theorem}

\begin{proof}
We use notation and results from \cite{BHORS}.
Let $P_{3,6}$ denote the convex cone of nonnegative rational sextics
and $\Sigma_{3,6}$ the subcone of sextics that are sums of squares of cubics
over $\RR$. The dual cone $\Sigma_{3,6}^\vee$ consists of sextics $f$
whose middle catalecticant $C(f)$ is positive semidefinite.
Its subcone $P_{3,6}^\vee $ is spanned by sixth powers 
of linear forms. It is known as the {\em Veronese orbitope}.
The difference $\Sigma_{3,6}^\vee \backslash P_{3,6}^\vee$
is a full-dimensional semialgebraic subset of $\RR[x,y,z]_6$.

We claim that general sextics $f$ in that set have real rank $\geq 11$.
Let $f $ be a general sextic in $\Sigma_{3,6}^\vee \backslash P_{3,6}^\vee$.
Suppose that $\rrk(f) = 10$. The middle catalecticant $C(f)$ is
positive definite. Proposition~\ref{lem:2468} tells us that the signature of
any representation  (\ref{eq:waring}) is $(10,0)$. This means that
$f$ lies in the Veronese orbitope $P_{3,6}^\vee$. This is a contradiction
to the hypothesis, and we conclude $\rrk(f) \geq 11$.
Using \cite[Theorem 1.1]{BBO}, this means that $11$ is a typical rank.

Consider the algebraic boundary of the Veronese orbitope $P_{3,6}^\vee$.
One of its two components is the determinant of the catalecticant $C(f)$, which
is the algebraic boundary of  the spectrahedron $\Sigma_{3,6}^\vee$.  The other component 
is the dual of the Zariski closure of the set 
of extreme rays of $P_{3,6} \backslash \Sigma_{3,6}$. That  Zariski closure was 
shown in \cite[Theorem 2]{BHORS} to be equal to the Severi variety of rational 
sextics, which has codimension $10$ and degree $26312976$ in $\PP^{27}$.
Every generic boundary point of $P_{3,6}^\vee$ that is not in 
the spectrahedron $\Sigma_{3,6}^\vee$
represents a linear functional whose maximum over $P_{3,6}$
occurs at a point in the Severi variety.

A result of Choi, Lam and Reznick (cf.~\cite[Proposition 7]{BHORS}) states that 
every general supporting hyperplane of $P_{3,6}^\vee$ touches the Veronese
surface in precisely $10$ rays. Every form in the cone spanned
by these rays has real rank $\leq 10$. Consider the subset of $P_{3,6}^\vee$
obtained by replacing each of the $10$ rays by a small neighborhood.
This defines a full-dimensional subset of forms $f \in P_{3,6}^\vee$ 
that satisfy $\rrk(f) = 10$. By construction, this subset must intersect the boundary of $P_{3,6}^\vee$
in a relatively open set. Its Zariski closure is the hypersuface dual to the Severi variety.
We conclude that this dual is an irreducible component of 
$\partial_{\rm alg}( \mathcal{R}_6)$.
\end{proof}

\begin{remark} \rm
The same proof applies also for octics $(d=8)$,
ensuring that  the algebraic boundary 
$\partial_{\rm alg}( \mathcal{R}_8)$ exists.
Indeed, $R(8) = 15$ coincides with the size of the middle
catalecticant $f$, and we can conclude that
every octic in  $\Sigma_{3,8}^\vee \backslash P_{3,8}^\vee$
has real rank bigger than $15$.
However, for even integers $d \geq 10$, this
argument no longer works, because the generic
complex rank exceeds the size of the middle catalecticant.
In symbols, $R(d) > \binom{d/2+2}{2}$. New ideas are
needed to establish the existence of the hypersurface
$\partial_{\rm alg}( \mathcal{R}_d)$  for $d \geq 9$.
\end{remark}

We record the following upper bounds on the real ranks of general ternary forms.
 
\begin{proposition}
Let $m(d)$ be the maximal typical rank of a ternary form of degree $d$. 
Then $$ m(d)\,\,\leq \,\,  {\rm min} \bigl(  \,\binom{d+1}{2}-2, \,2 R(d) \,\bigr) . $$
In particular, typical real ranks for ternary sextics are between $10$ and $19$. 
\end{proposition}

\begin{proof}
The same argument as in  \cite[Proposition 6.2]{BBO} shows
 $m(d) \leq m(d -1)+d$. The binomial bound follows by induction.
The bound $2R(d)$ comes from \cite[Theorem 3]{BT}.
\end{proof}

The anti-polar construction in
(\ref{eq:antipolar1}) extends to $f$ of degree $d=6$ and $d=8$. We define
$$\BQ(f)(a,b,c) \,\,=\,\,{\rm det}\bigl(C(f+ \ell^d)\bigr) \,-\,
 {\rm det}\bigl(C(f) \bigr) \qquad \text{   for   } \quad \ell=ax+by+cz.$$

\begin{lemma}\label{Blekhermangeneral}
Let $f$ be a general ternary form of degree $d\in\{4,6,8\}$ 
that is not in the cone $P_d^\vee$ spanned by $d$th powers.
  If $\BQ(f)$ has no real zeros than the real rank of $f$ exceeds $R(d)$.
  \end{lemma}

\begin{proof}
Suppose that $f$ is of minimal generic rank $R(d)$. Since $f$ is not a sum of $d$th powers of linear forms
over $\RR$, by Proposition \ref{lem:2468}, there exists a real linear form $\ell = ax + by + cz$ such that the catalecticant matrix of
$ f + \ell^d$  is degenerate; hence $\BQ(f)(\ell)=-{\rm det}\bigl(C(f) \bigr)$. On the other hand, there exists $\ell'$ such that the catalecticant matrix of $-f+\ell'^d$ also drops rank, so $\BQ(f)(\ell')=-\BQ(-f)(\ell')={\rm det}\bigl(C(f) \bigr)$.
Hence the real curve defined by $\BQ(f)$ is non-empty.
\end{proof}

Let $f \in \RR[x,y,z]_6$ be general, with $\crk(f) = R(6) = 10$.
In what follows we derive the algebraic representation of the K3
surface ${\rm VSP}(f)$ and its semialgebraic subset ${\rm SSP}(f)_\RR$.
 The apolar ideal 
$f^\perp$ is generated by nine quartics. The minimal free
resolution of $f^\perp$ equals $$
S \rightarrow S(-5)^9 \stackrel{A}{\longrightarrow}  S(-4)^9 \rightarrow S \rightarrow 0.
$$
Here $A$ is a skew-symmetric  $9 {\times} 9$-matrix with linear entries.
Its $8 {\times} 8$ subpfaffians generate  $f^{\perp}$. We write $A = A_1 x {+} A_2 y {+} A_3 z $
where $A_i \in
\bigwedge^2 f^\perp_4 \simeq \bigwedge^2 \RR^9$.
The variety ${\rm VSP}(f)$ is a K3 surface
of genus~$20$ and degree $38$. See \cite{Muk92b, RS}
for  details and proofs. The following representation, found 
in \cite[Theorem 1.7 (iii)]{RS}, is analogous to 
Proposition \ref{prop:Mukai35} and Lemma~\ref{prop:Mukai47}.

\begin{proposition}
The surface  ${\rm VSP}(f)$ is the intersection
of the Grassmannian ${\rm Gr}(5,9)$, in its Pl\"ucker 
embedding in  $\PP(\bigwedge^5 f^\perp_4) \simeq \PP^{125}$,
with the $20$-dimensional linear subspace
$$ \PP^{20}_A \,\,= \,\, \bigl\{\,
U \in \PP^{125} \,:\, U \wedge A_1 = U \wedge A_2 = U \wedge A_3 = 0\, \bigr\}.
$$
Inside this space, ${\rm VSP}(f)$ is cut out by
$153$ quadrics, obtained from the Pl\"ucker relations.
\end{proposition}

For any sextic $f $, we can compute
the surface ${\rm VSP}(f)$ explicitly, by the method explained
for quadrics in Example \ref{ex:x^2+y^2-z^2}.
Namely, as in (\ref{eq:affinecoord}),
we introduce local coordinates on ${\rm Gr}(5,9)$.
The equations defining $\PP^{20}_A$ translate into
quadrics in the $20$ local coordinates.
In analogy to (\ref{eq:wenowtransform}),  we transform the
$9 \times 9$-matrix $A$ into
\begin{small}$ \begin{pmatrix}   \star  & T \,\, \\       -T^t & 0 \,\, \end{pmatrix}$\end{small}.
Here, $T$ is a $5 \times 4$-matrix of linear forms 
whose $4 \times 4$ minors define the ten points in $\PP^2$ in
 the representation $f = \sum_{i=1}^{10} \ell_i^6$.
We can study ${\rm SSP}(f)_\RR$  and its boundary
inside the real K3 surface ${\rm VSP}(f)_\RR$ by means of the
hyperdeterminant for $m=4$ in Corollary \ref{cor:schlafli}.
The following example demonstrates this.
\\
\begin{example}  \rm
Consider  the sextic ternary $\,f=\sum_{i+j+k=3} (ix+jy+kz)^6$, where $i,j,k\in\ZZ_{\geq 0}$.
It is is given by a real rank $10$ representation.
 Consider the $9\times 9$ matrix of linear forms in some minimal  free resolution of
 the apolar ideal $f^{\perp}$.
 We transform this matrix into
  $$
A \,\, = \,\,
\begin{bmatrix}
 0 & l_1 & l_2 & l_3 & l_4 & l_5 & l_6 & l_7 & l_8 \\
-l_1 & 0 & l_9 & l_{10} & l_{11} & l_{12}  & l_{13} & l_{14} & l_{15}\\
-l_2 & -l_9 & 0 & l_{16} & l_{17} & l_{18} & l_{19} & l_{20} & l_{21}\\
-l_3 & -l_{10} & -l_{16} & 0 & l_{22} & l_{23} & l_{24} & l_{25} & l_{26}\\
-l_4& -l_{11} & -l_{17} & -l_{22} & 0 & l_{27} & l_{28} & l_{29} & l_{30} \\
-l_5 & -l_{12} & -l_{18} & -l_{23} & -l_{27} & 0 & 0 & 0 & 0\\
-l_6 & -l_{13} & -l_{19} & -l_{24} & -l_{28} & 0 & 0 & 0 & 0\\
-l_7& -l_{14} & -l_{20}& -l_{25} & -l_{29} & 0 & 0 & 0& 0\\
-l_8 & -l_{15} & -l_{21} & -l_{26} & -l_{30} & 0 & 0 & 0 & 0 
\end{bmatrix},
$$
\noindent where \begin{small} 
$\,l_1= \frac{6885}{631}x-\frac{4050}{631}y-\frac{175770}{631}z,\, l_2=-\frac{4050}{631}x+\frac{3240}{631}y+\frac{37665}{631}z\, l_3=-\frac{810}{631}z, \,l_4=\frac{324}{631}x+\frac{810}{631}y+\frac{2025}{631}z,\, l_5= -\frac{5}{631}x , \,l_6=\frac{21}{631}x, \,
 l_7= -\frac{4}{631}x, \,
  l_8= 0, \,
  l_9= -\frac{67230}{631}x-\frac{67230}{631}y-\frac{2791260}{631}z , \,
  l_{10}= \frac{3240}{631}x-\frac{4050}{631}y+\frac{37665}{631}z, \,
  l_{11}= -\frac{55728}{631}x-\frac{70308}{631}y-\frac{13446}{631}z,\,
   l_{12}= \frac{25}{631}x-\frac{25}{631}z,\,
    l_{13}= -\frac{1482}{631}x+\frac{818}{631}y+\frac{251}{631}z,\,
     l_{14}=\frac{668}{631}x-\frac{852}{631}y+z,\,
      l_{15}= \frac{10}{631}y-\frac{20}{631}z,\,
       l_{16}=-\frac{4050}{631}x+\frac{6885}{631}y-\frac{175770}{631}z, \,
       l_{17}=\frac{70308}{631}x+\frac{55728}{631}y+\frac{13446}{631}z , \,
       l_{18}=-\frac{10}{631}x+\frac{20}{631}z, \,
       l_{19}=\frac{852}{631}x-\frac{668}{631}y-z,\,
        l_{20}=-\frac{818}{631}x+\frac{1482}{631}y-\frac{251}{631}z, \,
        l_{21}=-\frac{25}{631}y+\frac{25}{631}z, \,
        l_{22}= -\frac{810}{631}x-\frac{324}{631}y-\frac{2025}{631}z,\, 
        l_{23}=0,\,
         l_{24}= \frac{4}{631}y  , \,
         l_{25}=-\frac{21}{631}y,\,
          l_{26}=\frac{5}{631}y , \,
          l_{27}= \frac{5}{631}z ,\,
            l_{28}=   -\frac{430}{631}z,\,
              l_{29}=  \frac{430}{631}z,\,
           l_{30}=  -\frac{5}{631}z $. \end{small} \noindent The upper right $5 \times 4$-submatrix of $A$
drops rank precisely on the 
ten points $(i:j:k) \in \PP^2$
where $i+j+k=3$ in nonnegative integers.

We introduce local coordinates on ${\rm Gr}(5,9)$ as follows. Let $U$ be the row span of $\begin{pmatrix} \,{\rm Id}_5 \!& \! V \end{pmatrix},$ where $ V=(v_{ij}) $  is a $5\times 4$ matrix of unknowns. We transform $A$ into the coordinate system given by
$U$ and its orthogonal complement:
$$
\begin{pmatrix}   \star & T \,\, \\       -T^t & 0\,\,
\end{pmatrix} \quad = \quad 
\begin{pmatrix}\, {\rm Id}_5 & V \ \\
                          \,      V^t &  \! - {\rm Id}_4 
  \end{pmatrix} \cdot
A \cdot \begin{pmatrix} \,{\rm Id}_5 & V\ \\
                           \,     V^t &  \! - {\rm Id}_4 
  \end{pmatrix}.
$$

We proceed as in Example \ref{ex:x^2+y^2-z^2}.
The lower  right $4\times 4$ block is zero whenever the corresponding $18$ quadrics in the $20$ local coordinates vanish. The submatrix $T$ is a $3\times 4\times 5$ tensor. Its hyperdeterminant ${\rm Det}(T)$ is
the discriminant for our problem. By Corollary~\ref{cor:schlafli},
this is a polynomial of degree $120$ in the entries of $T$.
The specialization of ${\rm Det}(T) $ to the $20$
local coordinates has degree $\leq 240$.
That  polynomial represents the algebraic boundary of 
${\rm SSP(f)}_\RR$ inside ${\rm VSP}(f)_\RR$,
similarly to Corollary~\ref{cor:inaffine}.
\hfill $ \diamondsuit$
\end{example}
\smallskip

In the paper we focused on general ternary forms. 
Special cases are also very interesting:

\begin{example} \rm  
Consider the monomial  $f = x^2y^2z^2$. 
By \cite{COV}, this has $\rrk(f) \leq 13$ because
$$
\begin{matrix} 360 f & =  & 
4 (x^6+y^6+z^6) + 
 (x+y+z)^6 + (x+y-z)^6 + (x-y+z)^6 + (x-y-z)^6 \smallskip \\ & & 
 -2 \bigl[ (x+y)^6 +(x-y)^6 +(x+z)^6+(x-z)^6 + (y+z)^6 +(y-z)^6 \bigr] .
\end{matrix}
$$
The apolar ideal is $f^\perp = \langle x^3, y^3, z^3 \rangle$. 
The radical ideal generated by general cubics in $f^\perp$,
\begin{equation}
\label{eq:twocubics}
C_1\,=\,\alpha x^3+ \beta y^3 + \gamma z^3 \quad {\rm and} \quad
C_2\,=\,\alpha' x^3+\beta' y^3 +\gamma' z^3,
\end{equation}
proves that $\,\crk(f)=9$. We can replace
$C_1$ and $C_2$ by two linear combinations
that are binomials, say $C'_1 = x^3  + \delta z^3$
and $C'_2 = y^3 + \epsilon z^3$.
For any choice of $\delta$ and $\epsilon$, at most three of the
nine points  of $V(C'_1 , C'_2 ) \subset \PP^2$ are real.
 This shows $\,\rrk(f)\geq 10$. 
 
We next prove  $\rrk(f) \geq 11$.  Assume that
$p_1,p_2,\ldots,p_{10} \in \PP^2_\RR $
give a rank $10$ decomposition  of $f $.
Consider a pencil of cubics $C_1+ t\cdot C_2$ passing through $p_1,\ldots, p_8$. 
We may assume $p_{10} \not\in V(C_1,C_2)$ and $C_2(p_9)=0$. 
 We claim that $C_2$ vanishes also at $p_{10}$.
 Indeed, $C_2$ acts by 
 differentiation and gives $C_2 (f)= \lambda_{10} C_2(p_{10}) \ell_{10}^3$
where $\lambda_{10} \in \RR^*$ and $C_2(p_{10})$ is the evaluation at $p_{10}$. 
However, $C_2(f) = C_2(x^2 y^2 z^2)$ contains none of 
the pure powers $x^3, y^3,z^3$ and so $C_2$ passes through $p_{10}$ as well. We now know that $C_1$ vanishes at neither
$p_9$ nor $p_{10}$. Differentiation gives
$$
C_1(f)\,\,=\,\,\alpha \ell_9^3 \,+\, \beta \ell_{10}^3, 
\qquad \hbox{where
 $\alpha=\lambda_9 C_1(p_9), \beta=\lambda_{10} C_1(p_{10})\,$ and $
 \,\lambda_9, \lambda_{10} \in \RR^*$.}
$$
Let $\ell_9=a_9x+b_9y+c_9z$ and $\ell_{10}=a_{10}x+b_{10}y+c_{10}z$.
The coefficients of $x^3,y^3,z^3$ in the cubic $C_1(f)$ vanish. Hence 
$$
\alpha a_9^3 +\beta a_{10}^3 \,=\,
\alpha b_9^3 + \beta b_{10}^3\,=\,
\alpha c_9^3 + \beta c_{10}^3\,=\,0.
$$
Over $\mathbb R$, these equations imply 
$a_9=-(\beta/\alpha)^{1/3}a_{10}$, $b_9=-(\beta/\alpha)^{1/3}b_{10}$, $c_9=-(\beta/\alpha)^{1/3}c_{10}$. So, $p_9$ and $p_{10}$ are the same point in $\mathbb P^2$. This is a contradition and we conclude $\rrk(f) \geq 11$. 

At present we do not know whether the real rank of $f$ is $11, 12$ or $13$.
The argument above can be extended to establish the following result:
{\em if $\rrk(f) \leq 12$ then there exists a decomposition {\rm (\ref{eq:waring})}
whose points $(a_i:b_i:c_i)$ all lie  on the Fermat cubic $\,V(x^3+y^3+z^3)$.}

Let us explore the possibility 
 $\rrk(f) = 12$. Then it is likely that the
$(a_i:b_i:c_i)$ are  the complete intersection
of a cubic and a quartic. They can be  assumed to have the form:
\begin{equation}
\label{eq:realrootclassification}
x^3 + y^3 + 1  \,\,=  \,\,ax^4 + bx^3y + cx^3 + dx + ey + 1  \,\,=\,\,  0.
\end{equation}
Hence, determining the real rank of $f$ leads directly  to the following
easy-to-state  question: \\
{\em Can we find real constants $a,b,c,d,e$ such that all
$12$ solutions to the equations {\rm (\ref{eq:realrootclassification})} are real?}
If the answer is ``yes'' then we can conclude  $\rrk(f) \leq 12$.
Otherwise, we cannot reach a conclusion. A systematic approach 
to this Real Root Classification problem is via the
discriminant of the system  (\ref{eq:realrootclassification}).
This discriminant is a polynomial of degree $24$ in $a,b,c,d,e$.
We would need to explore the connected components of the complement
of this hypersurface in $\RR^5$. For further reading on the rank geometry of
monomials we refer to \cite{BBT, CCG}.
\hfill $\diamondsuit$
\end{example}

\medskip

\begin{small}

\noindent
{\bf Acknowledgements.}\smallskip \\
We are grateful to  Greg Blekherman  for his
help with this project. We also thank
 Giorgio Ottaviani and Frank-Olaf Schreyer for valuable discussions.
  Bernd Sturmfels was supported by the
  US National Science Foundation (DMS-1419018)
  and the Einstein Foundation Berlin. Mateusz Micha{\l}ek is a PRIME DAAD fellow and acknowledges the support of Iuventus Plus grant 0301/IP3/2015/73 of the Polish Ministry of Science.

\end{small}

\medskip

\begin{small}

\end{small}

\bigskip

\noindent
\footnotesize {\bf Authors' addresses:}

\smallskip

\noindent Mateusz Micha{\l}ek,
Freie Universit\"at Berlin, Germany;
Polish Academy of Sciences, Warsaw, Poland,
{\tt wajcha2@poczta.onet.pl}

\smallskip

\noindent Hyunsuk Moon,
  KAIST,
     Daejeon, South Korea, {\tt octopus14@kaist.ac.kr}

\smallskip

\noindent Bernd Sturmfels, 
University of California, Berkeley, USA,
{\tt bernd@berkeley.edu}

\smallskip

\noindent Emanuele Ventura,
Aalto University,  Finland,
{\tt emanuele.ventura@aalto.fi}

\end{document}